\documentclass[11pt,a4paper,oneside,reqno]{amsart}
\newtheorem{theorem}{Theorem}[section]
\newtheorem{lemma}[theorem]{Lemma}
\usepackage[numbers]{natbib}
\usepackage{setspace}
\usepackage{tikz-cd}
\theoremstyle{definition}
\newtheorem{definition}[theorem]{Definition}

\newtheorem{proposition}[theorem]{Proposition}
\newtheorem{corollary}[theorem]{Corollary}
\newtheorem{notation}[theorem]{Notation}
\newtheorem{fact}[theorem]{Fact}
\newtheorem{question}[theorem]{Question}
\newtheorem{remark}[theorem]{Remark}
\newtheorem{Notation and Remark}[theorem]{Notation and Remark}
\newtheorem{convention}[theorem]{Convention}

\usepackage{amsmath}
\usepackage[inline]{enumitem}
\usepackage{amssymb}
\usepackage{lipsum}
\usepackage{latexsym}
\usepackage{natbib}
\usepackage{fullpage}
\usepackage{xspace} 
\DeclareMathOperator{\alfa}{\alpha}

\newcommand{\Ralfak}[1][R]{{#1}[\alfa_1,\ldots,\alfa_n]}

\DeclareMathOperator{\im}{Im}

\newcommand{\PrPol}[1][R]{\mathcal{P}(#1)}
\newcommand{\PolFun}[1][R]{\mathcal{F}(#1)}

\DeclareMathOperator{\quv}{\xspace{ }\triangleq\xspace{ }}

\usepackage{mathtools}
\newcommand{\UPF}[1][R]{\mathcal{F}(#1)^{\times}}
\usepackage{epsf}
\usepackage{layout}
\usepackage{hyperref}  
\usepackage{comment}    
               
\begin{document}
	\openup .3em
	\title[]{On the structures of a monoid of triangular vector- permutation polynomials, its group of units and its induced group of permutations}
	
	\author{Amr Ali Abdulkader Al-Maktry}
	\address{\hspace{-12pt}Department of Analysis and Number Theory (5010) \\
		Technische Universit\"at Graz \\
		Kopernikusgasse 24/II \\
		8010 Graz, Austria}
	\curraddr{}
	\email{almaktry@math.tugraz.at}
	\curraddr{}
	\email{}
	
	

	\date{}
	
	\dedicatory{}
	\maketitle
	
	\section*{Abstract}
		Let $n>1$ and let $R$ be a   commutative ring with identity  $1\ne 0$ and $R[x_1,\ldots,x_n]^n$ the set of all $n$-tuples of polynomials of the form $(f_1,\ldots,f_n),$ where $f_1,\ldots,f_n\in R[x_1,\ldots,x_n]$. We call these $n$-tuples  vector-polynomials. We define composition on $R[x_1,\ldots,x_n]^n$ by $$\vec{g}\circ \vec{f}=(g_1(f_1, \ldots ,f_n), \ldots ,g_n (f_1, \ldots ,f_n)), \text{ where }\vec{f}=(f_1, \ldots ,f_n), \vec{g}=(g_1, \ldots ,g_n).$$
	In this paper, we investigate  vector-polynomials of the form $$ 	\Vec{f}=(f_0,f_1 +x_2g_1,\ldots,	f_{n-1} +x_n g_{n-1}),$$ where $f_0\in R[x_1]$ permutes the elements of $R$ and $f_i ,g_i\in R[x_1,\ldots,x_i]$ such that each $g_i$ maps $R^i$ into the units of $R$  ($i=1,\ldots, n-1$). We show that each such  vector-polynomial permutes the elements of $R^n$ and  that the set of all such vector-polynomials   $\mathcal{MT}_n$  is a monoid with respect to composition. We also show that $ 	\Vec{f} $  is invertible in $\mathcal{MT}_n$ if and only if $f_0$ is an $R$-automorphism of $R[x_1]$ and $g_i$ is invertible in $R[x_1,\ldots,x_i]$ for $i=1,\ldots, n-1$. 
	When $R$ is finite, the monoid $\mathcal{MT}_n$ induces a finite group of permutations of $R^n$. Moreover, we  decompose the monoid $\mathcal{MT}_n$ into an iterated semi-direct product of $n$ monoids. Such a decomposition allows us to obtain 
	similar decompositions of its group of units  and, when $R$ is finite, of its induced  group of permutations. 
	Furthermore, the decomposition of the induced group helps us to characterize some of its properties.

	\section{Introduction}\label{CHSIntro} 
	
	Let $R$ be a   commutative ring (not necessarily finite) with unity and let $R[x_1,\ldots,x_n]$ be its polynomial ring in $n$ variables. A vector-polynomial $\vec{f}$ is an $n$-tuple $(f_1,\ldots,f_n)$ of polynomials over $R$ in the variables $x_1,\ldots,x_n$, i.e., $\vec{f}=(f_1,\ldots,f_n)$ with $f_i\in R[x_1,\ldots,x_n]$ for $i=1,\ldots,n$. It is well- known that every vector-polynomial $\vec{f}$ defines an $R$-endomorphism $\phi_{\vec{f}}$  of $R[x_1,\ldots,x_n]$
	by $\phi_{\vec{f}}(g)=g(f_1,\ldots,f_n)$  for any $g\in R[x_1,$\ldots$,x_n]$. Furthermore, 
	$\vec{f}$ is an $R$-automorphism of $R[x_1,$\ldots$,x_n]$ if and only if it is invertible with respect to composition. The set of all these  automorphisms is a group, which is denoted by $GA_n$. Vector-polynomials are commonly referred to as 
	polynomial maps in the literature.
	We are interested not only in the vector-polynomial (polynomial map) but also in its induced function 
	(map) on $R^n$ via evaluation. This distinction helps avoid confusion between an $n$-tuple of polynomials 
	and a single polynomial in this context.
		
	Several results have been achieved  concerning these $R$-endomorphisms of $R[x_1,\ldots,x_n] $,   mainly considering  the group $GA_n$, its subgroups,  and its elements. 
	The developed ideas cover several aspects of investigations, including the Jacobian Conjecture
	and equivalent statements (see for example and the references therein~\cite{Autjac}), the inclusion of the tame group (see for example~\cite{Jung,Kulktame2vgencase,Umir3var}), and the study of structures of some subgroups of $GA_n$ (see \cite{Trgroup,Jakunormgrp}).

	Previous works have primarily focused  on the case where $R$ is an infinite field with characteristic zero. 
	Nevertheless,   researchers have also considered  finite fields to investigate similar questions (see~\cite{Maubffield,Maubdrecks,Leshchenkounitrfinitef}), disproving some famous results known to be true in  the infinite case (see~\cite{Hakutdrecks,Maubdrecks}), and  reformulating some others (for instance~\cite{MaubachJscform}). Notably, the permutations induced by the elements of $GA_n$ on $R^n$ have
	a powerful impact.

	In this paper, we  explore the structure of a monoid of triangular vector-polynomials with 
	a specific property. We consider the monoid of all triangular vector-polynomials of the form $(f_0 ,f_1 +x_2g_1,\ldots,f_{n-1} +x_ng_{n-1})$ with the condition that $f_0$ is a permutation polynomial on $R$ and 
	$g_i$ is a unit-valued on $R $ for $i=1,\ldots, n-1$. The elements of this monoid define permutations of $R^n$. Additionally, we observe that the group of units of this monoid generally contains properly  the (lower) Jonqui\`eres (classical) triangular group unless  $R$ is a reduced ring. In a reduced ring the invertible polynomials coincide with the units of the ring $R$, and in this case,     the group of the triangular monoid is just the Jonqui\`eres group (see Remark~\ref{coincidestrctures}).
	
	The monoid under consideration can be decomposed into an iterated semi-direct product of $n$ monoids by iteration (see Theorem~\ref{mondec}). 	
	Surprisingly, this decomposition admits a similar decomposition
	of the group of units of the triangular monoid into an iterated semi-direct product of $n$ groups by finding the units of those factor monoids in the first decomposition  (Theorem~\ref{decoftg}). Further, when $R$ is finite, the triangular monoid induces a group of permutations on $R^n$, which  is its homomorphic image via
	the natural epimorphism $\pi_{n}$ (defined in Remark~\ref{eppi}). Again, we obtain an iterated semi-direct product decomposition  of the induced
	group just by taking images (via $\pi_{n}$) of the factor monoids in that decomposition of  the triangular monoid (Theorem~\ref{ddirec}).
	Such decomposition allows us to characterize the  solvability and the nilpotency of this group restricted (by the  Chinese Reminder Theorem) to the local case (Theorem~\ref{infucedgrproperty}).
	
	It is remarkable that in the case where $R$ is an algebraically closed field of characteristic zero,  the triangular monoid, its group of units, and its induced group all coincide with the Jonqui\`eres triangular group. This is because unit-valued polynomials are just the non-zero elements  of $R$, and there is no non-linear permutation polynomial on $R$  (see Proposition~\ref{algcase}).
	This motivates us to
	suggest a modification for the tameness problem in the general case by replacing the Jonqui\`eres with our triangular monoid, and
	then posing some related questions (see Section~\ref{sec6}).

	Having  previously worked on polynomial functions and polynomial permutations over finite commutative rings (in one variable) rather than polynomial mappings (in several variables), we can ensure that our motivation to study the triangular monoid emanates from a simple question on polynomial permutations. In other words, given a finite ring $R$ and $R_n$ is its dual numbers extension in $n$ variables. Then, does every polynomial permutation on $R_n$ can  be viewed as a vector-permutation polynomial on the  base ring $R$? More generally, can the group of permutation polynomials on $R_n$  be viewed as a group of permutations of $R^{n+1}$ induced vector-polynomials in $n+1$ variables? We answer this question affirmatively and provide more details (see Section~\ref{sec5}).
	
	The structure of our paper is as  follows: in Section~\ref{sec2}, we introduce  our notation and basic definitions, and we recall some facts about polynomials and polynomial functions. In section~\ref{sec3}, we 
	describe the elements of the triangular monoid  and its units and also obtain some simple relations for them. Section~\ref{sec4} consists of three subsections, the first subsection contains the decomposition of the triangular monoid into an iterated semi-direct product  of monoids, while the other two subsections are devoted to 
	the decompositions of its group of units and its induced group of permutations, respectively. In section~\ref{sec5}, we consider the problem of embedding the monoid of permutation polynomial (in one variable) of   certain ring extensions of the ring $R$ into the triangular monoid. Finally, in Section~\ref{sec6}, we formulate the tameness problem and suggest some related questions. 
	\section{Preliminaries}\label{sec2}
		Throughout this paper, let $R$ be a  commutative ring with $1\ne0$, $R^\times$ its group of units, and   whenever $R$ is  a finite local ring, $M$ will denote its unique maximal ideal. Also, we use $\mathbb{F}q$ to denote its residue field $R/M$ or, more generally, a finite field of $q$ elements.
	
	Let $n\ge 1$. By $n$, we denote the number of variables $x_1,\ldots,x_n$; and 
	to distinguish the elements of $R[x_1,\ldots,x_n]^n$ from those of $R[x_1,\ldots,x_n]$, we use the symbol $\Vec{f}$ to indicate an  $n$-tuple  $(f_1,\ldots, f_n)$ of polynomials $f_1,\ldots, f_n\in R[x_1,\ldots,x_n] $. In a similar manner we use symbols like $\Vec{F}$ to indicate function $\Vec{F}\colon  R^n\longrightarrow R^n$, or equivalently to indicate an $n$-tuple $(F_1,\ldots,F_n)$ of functions $F_1,\ldots,F_n\colon
	R^n \longrightarrow R$,  i.e. $\Vec{F}=(F_1,\ldots,F_n)$.
	It is well-known that a vector-polynomial $\vec{f}=(f_1,\ldots,f_n)$ defines an $R$-endomorphism of $R[x_1,\ldots,x_n]$:    $\phi_{\vec{f}}\colon R[x_1,\ldots,x_n] \longrightarrow  R[x_1,\ldots,x_n]$, $\phi_{\vec{f}}(h)= h(f_1,\ldots,f_n)$.  Such an endomorphism is called an $R$-automorphism if there is a  vector-polynomial $\vec{g}=(g_1,\ldots,g_n)$ such that 
	$$\vec{f}\circ \vec{g}=(g_1(f_1, \ldots ,f_n), \ldots ,g_n (f_1, \ldots ,f_n))
	=(x_1,\ldots,x_n).$$ Here, we notice that the above operation is anti the normal composition. However, we will  adapt later the normal one (see Definition~\ref{compose}). The reason is that we are also interested in the (vector)-functions induced by (vector)-polynomials, and we require that the operation on polynomials and their functions  be consistent.  Also,  it is noticeable here that despite  any one of these  operations being under consideration in the monoid ${R[x_1,\ldots,x_n]}^n$,  the terms invertible, left invertible and right invertible are equivalent~\cite{Autjac}.
	
	Furthermore,  we note  that  if $\vec{F}$ is vector-function obtained from the $R$-automorphism $\vec{f}$	via evaluation, that is,
	\[	\vec{F}(a,\ldots,a_n)=\vec{f}(a_1,\ldots,a_n) =
	(f_1(a_1,\ldots,a_n), 	\ldots,f_n(a_1,\ldots,a_n))
	,\text{   for every }(a_1,\ldots,a_n)\in R^n,\]
	permutes the elements of $R^n$. However, the converse is not true. For example, consider the case $n=1$, $R=\mathbb{F}_q$,
	and $\vec{f}(x)=f(x)=x^q$. Evidently, $\vec{F}=F=id_{R^n}$; but, there is no  polynomial $g$ such that
	$f\circ g =x$.
	\begin{notation}
		For $k\ge 1$, let $R_{[k]}$ denote the polynomial ring   $R[x_1,\ldots,x_k]$ and let $R_{[k]}^{\times}$ denote its group of units. But, when $k=1$,
		we use 
		$R[x]$ (or $R[x_1]$) and   	$R[x]^\times$ (or $R[x_1]^\times$) to denote the polynomial ring in one variable and its group of units   rather than writing $R_{[1]}$ and $R_{[1]}^\times$, respectively.
	\end{notation}
	\begin{definition}\label{predef}	
		Let $R$ be a commutative ring, and  $f,g\in R_{[n]}$. Then:
		\begin{enumerate}
			\item The polynomial $f$ gives rise to a polynomial function $F\colon R^n\longrightarrow R$ by substitution for the variables, that is, $F(a_1,\ldots,a_n)=f(a_1,\ldots,a_n)$ for every $(a_1,\ldots,a_n)\in R^n$.
			Further, if $F$ maps $R$ into $R^\times$  the group of units of $R$, we call $f$ a unit-valued polynomial and $F$ is a unit-valued polynomial function (on $R$).
			\item If $g(a_1,\ldots,a_n)=f(a_1,\ldots,a_n)$ for every $(a_1,\ldots,a_n)\in R^n$, this means that $f$ and $g$ induce the same function $F$ on $R$ and we write $f \quv  g$ on $R$. In this case, we say $f$ and $g$ are equivalent on $R$.
			\item We denote by $\PolFun[R^n]$  the set of all polynomial functions (in  $n$ variables) on $R$.
			\item We denote by $UV(R_{[n]})$ the set of all unit-valued polynomials in $n$ variables.
					\end{enumerate}
	\end{definition}
	
	From Definition\ref{predef}, we have easily the following fact.
	\begin{fact}
		Let $R$ be a commutative ring.
		Then
		\begin{enumerate}
			\item $UV(R_{[n]})$ is a submonoid of $(R_{[n]},``\circ")$;
			\item $\PolFun[R^n]$ is a monoid with respect to composition;
			\item $\PolFun[R^n]$ is a commutative ring with  pointwise operation. Furthermore, $\PolFun[R^n]$ is finite whenever $R$ is finite, and in this case  the group of units of
			$\UPF[R^n]$ consists of all unit-valued polynomial functions.
		\end{enumerate}
	\end{fact}
	\begin{remark}\label{uni-v properties}
		\begin{enumerate}
			\item By Definition \ref{predef}, every unit polynomial in a polynomial ring is a unit-valued polynomial but not vice versa. For example, consider the case $R$ is a finite local ring and $f(x)=(x^q-x)+1$.
			\item When $R$ is a domain, the invertible polynomials are just the units of $R$.
			\item When $R$  contains  a non-zero nilpotent, it contains a   nilpotent $r\ne 0$ such that $r^2=0$.  But, then  the polynomial $f(x)=rx+1$ admits
			$rx-1$ as multiplicative inverse.
			\item Let $R$ be a commutative ring and $f(x)=\sum_{i=0}^{k} a_ix^i\in R[x]$. Then $f$ is a unit in $R[x]$
			if and only if $a_0$ is a unit of $R$ and $a_i$ is nilpotent for $i=1,\ldots, k$ (see for example \cite[Proposition.~1.3.1]{finiteringbook}).
		\end{enumerate}
	\end{remark}
	
	\begin{definition}
		Let $R$ be a finite commutative ring. We call the group $\UPF[R^n]$ the group unit-valued polynomial functions (in $n$ variables) on $R$.
	\end{definition}

	\begin{definition}\label{compose}
		Let \[ \Vec{f}=(f_1,\ldots,f_n),\text{ and }    \Vec{g}=(g_1,\ldots,g_n),  \text{{ where} } f_i,g_i \in R_{[n]} \text{ for }i=1,\ldots,n.\] Then  we define  $f_i(\vec{g})=f_i(g_1,\ldots,g_n)$ for $i=1,\ldots,n$; and
		\[\vec{f}\circ \vec{g}=(f_1(\vec{g}),\ldots, f_n(\vec{g}))=  (f_1(g_1,\ldots,g_n),\ldots,f_n(g_1,\ldots,g_n)).  \] 
	\end{definition}
	From the Definition~\ref{compose}, one easily prove the following fact.
	\begin{fact}\label{tricalc}
		For $i=1,\ldots, n$, let   $f_i\in  R_{[i]}$.  Then
		\begin{enumerate}
			\item $(f_1,f_2,\ldots,f_n )=(f_1,x_2,\ldots,x_n)\circ\cdots\circ(x_1,x_2,\ldots,f_n)$; 
			\item $(x_1,\ldots,x_i,f_{i+1},\ldots f_n)\circ \vec{g}=\vec{g}\circ (x_1,\ldots,x_i,f_{i+1}(\vec{g}),\ldots f_n(\vec{g}))$, where 
			
			$\vec{g}=(x_1, \ldots,x_{j-1},f_j,x_{j+1},\ldots,x_n)$ and $j\le i$.
				\end{enumerate}
	\end{fact}
	\begin{definition}\label{CHSequvfun}
		Let $R$ be a commutative ring, and  $f_1,\ldots,f_n;g_1,\ldots,g_n\in R_{[n]}$. Then
		\begin{enumerate}
			\item The  vector-polynomial $\vec{f}=(f_1,\ldots,f_n)$ gives rise to a  vector-polynomial function of $R^n$, $\vec{F}\colon R^n\longrightarrow R^n$ by substitution for the variables, that is,
			\[ \vec{F}(a_1,\ldots,a_n)=\vec{f}(a_1,\ldots,a_n)=(f_1(a_1,\ldots,a_n),\ldots,f_n(a_1,\ldots,a_n)),\text{   for every }(a_1,\ldots,a_n)\in R^n.\]
			If $\vec{F}$ is bijective, we call  $\vec{F}$ a vector-polynomial permutation,  $\vec{f}$ a vector-permutation polynomial and $f_i$ a permutation polynomial for $i=1,\ldots,n$.
			
			\item If $\vec{g}=(g_1,\ldots,g_n)$ and $\vec{g}(a_1,\ldots,a_n)=\vec{f}(a_1,\ldots,a_n)$ for every $(a_1,\ldots,a_n)\in R^n$, this means that $\vec{f}$ and $\vec{g}$ induce the same function $F$ on $R^n$ and we  write $\vec{f} \quv  \vec{g}$ on $R^n$.
			\item We denote by $\Vec{\mathcal{F} }(R^n)$  the set of all  vector-polynomial functions of $R^n$. 
				\end{enumerate}
	\end{definition}
	\begin{remark}
		\begin{enumerate}
			\item It is interesting to note that    every invertible vector-polynomial is a vector-permutation polynomial  and  that the converse is not always true;    however, in the case where $R$ is an algebraically closed field of characteristic zero, vector-permutation polynomials are just the invertible vector-polynomials (see \cite[Theorem 2.2]{AutAlgclosed}  and \cite[Page.~80]{Autjac}). 
			
			\item When $n>1$, a permutation polynomial $f\in R_{[n]}$ can not alone induce a permutation of $R^n$ by Definition~\ref{CHSequvfun}. However, when $n=1$, the concept of permutation polynomial and vector-permutation polynomial coincide.  Permutation polynomials in one variable and their induced permutations play a vital rule in this context. For this reason we collect the related  terminology in the following definition.
		\end{enumerate}  
	\end{remark}
	\begin{definition}\label{pertermin}
		Let $R$ be a commutative ring, and  let $f\in R[x]$ and $F\colon R \longrightarrow R$ be its induced function.
		Then 
		\begin{enumerate}
			\item $f$ is called a \emph{ permutation polynomial} whenever $F$ is a permutation of $R$. In this case, $F$ is called \emph{polynomial permutation}.
			\item By $\mathcal{MP}(R)$, we denote the monoid of permutation polynomials on $R$ consisting of all permutation polynomials.
			\item  The group of units of $\mathcal{MP}(R)$ consists of all $R$-automorphisms of $R[x]$, which we denote by $Aut_R(R[x])$.
			\item If $R$ is finite, the set of all polynomial permutations on $R$ is a finite group which we denote by $\PrPol$.
		\end{enumerate} 
	\end{definition}
	\begin{remark}\label{Gilmaut}
		In \cite{GilmmerAut}, Gilmer showed that a polynomial $f\in R[x]$ is an $R$ automorphism of $R[x]$ if and only if
		$f(x)=a_0+a_1x+\cdots+a_mx^m$ with $a_1$ is unit and $a_i$ is nilpotent for every $i\ge 2$. It follows, when   $R$ is reduced, i.e. $R$ has no non-trivial nilpotents elements, that $Aut_R(R[x])$ consists only of invertible linear polynomials.
	\end{remark}
	\begin{remark}\label{eppi}
				There is a natural  epimorphism $\pi_{n}\colon R_{[n]}^n\longrightarrow  \Vec{\mathcal{F} }(R^n)$ with respect to composition 
			which maps every vector-polynomial into its induced vector-polynomial function $\vec{F}$. 
		\end{remark}
	\begin{lemma}\label{uniqex}
		Let $x_1,\ldots, x_{k+1}$ be variables, and let $f,g, h,u\in R_{[k]}$ such that $h$ and $u$ are unit-valued polynomials. Then $f+x_{k+1}u \quv g+x_{k+1}h$  on $R$ if and only if 
		$f  \quv g $  and $ u \quv h$  on $R$.
	\end{lemma}
	\begin{proof}
		Suppose that  $f+x_{k+1}u \quv g+x_{k+1}h$  on $R$. Then setting $x_{k+1}=0$ yields that
		$f  \quv g $     on $R$. Hence,   $ x_{k+1}u \quv x_{k+1}h$  on $R$. But, then $x_{k+1}=1$
		implies that		
		$ u \quv h$  on $R $. The other implication is obvious.
	\end{proof}
	\begin{fact}\label{equivfac}
		Let $R$ be a finite commutative ring. Let $\vec{F}$ be the  vector-polynomial permutation of $R^n$   induced by the vector-polynomial $\vec{f}=(f_1,\ldots,f_n)$, where $f_1,\ldots,f_n\in R_{[n]}$. If the inverse of $\vec{F}$, $\vec{F}^{-1}$, is induced by  a vector-polynomial $\vec{g}=(g_1,\ldots,g_n)$, where $g_1,\ldots,g_n\in R_{[n]}$, then
		$f_i\circ g_i\quv x_i$ on $R$ for $i=1,\dots,n$.
	\end{fact}
	In the following, we recall the definition of  a short exact sequence of groups and its relation with the semi-direct product of groups.
	A finite sequence of groups and homomorphisms 
	$$1 \rightarrow A \xrightarrow{\alpha} B \xrightarrow{\beta} C \rightarrow 1$$ is a short exact sequence
	if $\im \alpha =\ker \beta$. Further, if $\alpha$ is the inclusion function, $i$, $B$ is an extension of $A$ by $C$. 
	This allows us to give a definition of the semi-direct product of groups by means of a short exact sequence. 
	\begin{definition}\cite[Page 760]{rotadv}\label{semidi}
		An extension  $$1 \rightarrow A \xrightarrow{i} B \xrightarrow{p} C \rightarrow 1$$ is 
		split if there exists a homomorphism $j\colon C \rightarrow B$ such that $pj=id_C$. In this case,
		the group $B$ is called the semi-direct product of $A$ by $C$, which is denoted by $A\rtimes C$.
	\end{definition} 
	
	\section{The monoid of triangular  vector-permutation polynomials}\label{sec3}
	In this section, we   give the construction of a monoid of triangular vector-polynomials. 
	Recall from Definition~\ref{predef} that $UV(R_{[k]})$ stands for the monoid of unit-valued polynomials in $k$ variables.	
	\begin{lemma}\label{closnesop}
		Let $R$ be a commutative ring. Let $\Vec{f} =
		(f_1,		f_2 +x_2g_2 ,\ldots,
		f_{n} +x_n g_{n})$
		and $\vec{h}=(h_1 ,h_2 +x_2w_2 ,\ldots,h_{n}  +x_n w_{n})$, where $f_1,h_1\in \mathcal{MP}(R)$, $f_i,h_i \in R_{[i-1]}$ and $u_i,w_i\in UV(R_{[i-1]})$ with $i=2,\ldots,n-1$.
		Then $
		\vec{f}\circ \vec{h}=(l_1,l_2+x_2r_2,\ldots,	l_{n}+x_n r_{n}),
		$
		for some $l_1\in \mathcal{MP}(R)$, $l_i\in R_{[i-1]}$ and  $r_i\in UV(R_{[i-1]})$  for $i=2,\ldots,n$.
	\end{lemma}
	\begin{proof}
		Let $\Vec{h}\in \mathcal{MT}_n$ and $\Vec{f}$ as above. Then, 		
		by Definition~\ref{compose},   
		\begin{equation}\label{prodre1}
			\vec{f}\circ \vec{h}=(l_1,l_2+x_2r_2,\ldots,	l_{n}+x_n r_{n}),
		\end{equation} where  \begin{equation}\label{prodre2}
			\begin{matrix}
				l_1(x_1)=f_1(h_1(x_1))\\
				l_i(x_1,\ldots,x_{i-1})=f_i(\vec{v_i})+h_i(x_1,\ldots,x_{i-1})g_i(\vec{v_i}),\\
				r_i(x_1,\ldots,x_{i-1})= w_i(x_1,\ldots,x_{i-1})g(\vec{v_i})\text{ and }  \\
				\vec{v_i}(x_1,\ldots,x_{i-1})=(h_1(x_1),
				h_2(x_1)+x_2w_1(x_1),\ldots,
				h_{i-1}(x_1,\ldots,x_{i-2})+x_{i-1}w_{i-1}(x_1,\ldots,x_{i-2})),\\
				i=2,\ldots,n.	
			\end{matrix}
		\end{equation} 
		But, then $l_1=f_1\circ h_1\in \mathcal{MP}(R)$ since $f_1,h_1\in \mathcal{MP}(R)$; and  for $i=2,\ldots,n$,	$r_i= w_ig(\vec{v_i})\in UV(R_{[i-1]})$ since $w_i,g(\vec{v_i})\in UV(R_{[i-1]})$.
	\end{proof}
	Recall from Remark~\ref{eppi} that $\pi_n$ is the epimorphism assigns to each vector-polynomial
  its induced function on $R^n$.
	\begin{lemma}\label{pervpol}
		Let $R$ be a commutative ring. Let $f_1\in \mathcal{MP}(R)$ and for $i=2,\ldots,n$ let
		$f_i \in R_{[{i-1}]} $  and   $g_i\in UV(R_{[{i-1}]]})$. Then the vector-polynomial:
				$\Vec{f}=
		(f_1 ,			f_2 +x_2g_2,\ldots,
		f_{n}+x_n g_{n})
		$
		induces a permutation of $R^n$, that is $\pi_n(\vec{f}) $ is a permutation of $R^n$.
	\end{lemma}
	\begin{proof}
		Set $\Vec{F}=\pi_n(\vec{f})$.	First, we show that $\Vec{F}$ is one-to-one. For this, let $a_i,b_i\in R$  for $i=1,\dots,n$ such that \[\Vec{F}(a_1,\ldots,a_n)=\Vec{f}(a_1,\ldots,a_n)=\Vec{f}(b_1,\ldots,b_n)=\Vec{F}(b_1,\ldots,b_n).\]
		Then, we have the following equations:
		\[
		\begin{matrix}
			f_1(a_1)=f_1(b_1)\\
			f_2(a_1)+a_2g_1(a_1)=f_2(b_1)+b_2g_1(b_1)  \\
					\vdots\\
			f_{n}(a_1,\ldots,a_{n-1})+a_n g_{n}(a_1,\ldots,a_{n-1})=f_{n}(b_1,\ldots,b_{n-1})+b_n g_{n}(b_1,\ldots,b_{n-1})
		\end{matrix}
		\] 
		Thus, from the first equation, since $f_1$ is a permutation polynomial on $R$, we have  $a_1=b_1$.
		Then, the second equation becomes $ a_2g_2(a_1)= b_2g_2(a_1)$. But, then $a_2=b_2$ as $g_2(a_1)$ is
		a unit  since $g_2$  is a unit-valued polynomial on $R$. Continuing in this manner, we find that
		$a_i=b_i$, $1\le i\le n-1$ and the $i+1$th equation is reduced to 
		$ a_{i+1} g_{i+1}(a_1,\ldots,a_{i})=b_{i+1} g_{i+1}(a_1,\ldots,a_{i})$. Hence, $a_{i+1}=b_{i+1}$ since  $g_{i+1}$ is a unit-valued polynomial. Therefore,  $a_i=b_i$ for $i=1,\dots,n$. 
		To show $\Vec{F}$ is surjective, let $(c_1,\ldots,c_n)\in R^n$ be arbitrary. Then, we solve the equation   \[ \Vec{f}(x_1,\ldots,x_n)= (c_1,\ldots,c_n).\]
		Therefore, we have the following system of $n$ equations:		
		\[		\begin{matrix}
			f_1(x_1)=c_1\\
			f_2(x_1)+x_2g_2(x_1)=c_2  \\
			\vdots\\
			f_{n}(x_1,\ldots,x_{n-1})+x_n g_{n}(x_1,\ldots,x_{n-1})=c_n.
		\end{matrix}
		\] Since $f_1$ is a permutation on $R$, $x_1=a_1=F_1^{-1}(c_1)$  is the   solution of the 1st equation, where $F_1^{-1}$ stands for the inverse permutation of the permutation obtained by $f_1$ on $R$. So, the second equation becomes $f_2(a_1)+x_2g_2(a_1)=c_2$. Since $g_2$ is a unit-valued on $R$,
		$g_2(a_1)$ is invertible in $R$. Thus, $x_2=a_2=(c_2-f_2(a_1)) {g_2(a_1)} ^{-1}$. Solving these equations and the fact that $g_i$ is a unit valued on $R$  yield  \[x_i=a_i=(c_{i-1}-f_{i}(a_1,\ldots,a_{i-1})) {g_{i}(a_1,\ldots,a_{i-1})} ^{-1}\text{, }i=2,\ldots,n. \qedhere\]
	\end{proof}
	\begin{theorem}\label{constmon}\label{Constgr}
		Let $R$ be a  commutative ring with unity and   $n>1$  an integer. Let $\mathcal{MT}_n$ be the set of all vector-polynomials of the form: 
		\begin{equation}\label{vectel}		
			\Vec{f}(x_1,\ldots,x_n)=
			(f_1(x_1),			f_2(x_1)+x_2g_2(x_1),\ldots,
			f_{n}(x_1,\ldots,x_{n-1})+x_n g_{n}(x_1,\ldots,x_{n-1})),
		\end{equation}
		where   $f_1$ is a permutation polynomial
		on $R$, $f_i\in R_{[{i-1}]} $ and   $g_i\in UV(R_{[{i-1}]]})$   for $i=2,\ldots, n$. Then $\mathcal{MT}_n$ is a monoid of vector-permutation polynomials with respect to composition.
		Moreover, if $R$ is finite, then $\pi_n(\mathcal{MT}_n) $ is a group of permutations of $R^n$.
		\end{theorem}
	\begin{proof}
			By Lemma~\ref{pervpol}, $\mathcal{MT}_n$ consists of vector-permutation polynomials.
		To show $\mathcal{MT}_n$ is a monoid with respect to the composition $``\circ"$ we only need to show  that $\mathcal{MT}_n$ is closed with respect to $``\circ"$ 	since the identity  $  (x_1,x_2,\ldots,x_n)\in \mathcal{MT}_n$ and  composition is always associative (see Definition~\ref{compose}).  
		But, this is the case by Lemma~\ref{closnesop}.
		
		Moreover, if $S_{R^n}$ is the symmetric group on  $R^n$, then the first part of the proof shows that   $\pi_n(\mathcal{MT}_n)\subseteq S_{R^n} $.
		Since we consider the case $R$ is finite, $\pi_n(\mathcal{MT}_n)$ is a finite subset of the finite group
		$S_{R^n}$. Thus, we need only show that $\pi_n(\mathcal{MT}_n)$ is closed with respect to composition.
		Now, let $\vec{F},\vec{G}\in \pi_n(\mathcal{MT}_n)$. Then there exist $\vec{f},\vec{g}\in  \mathcal{MT}_n$ such that $\vec{F}=\pi_n(\vec{f})$ and $\vec{G}=\pi_n(\vec{g})$. Therefore,
		\[\vec{F}\circ\vec{G}=\pi_n(\vec{f})\circ\pi_n(\vec{g})=\pi_n(\vec{f} \circ \vec{g})\in \pi_n(\mathcal{MT}_n).\qedhere\]
	\end{proof}
	As a direct result, we have the following corollary.
	\begin{corollary}
		Let $R$ be a finite commutative ring.	Let $k>1$ and $f,g\in R_{[k]}$. Suppose that 
		$f$ is  unit-valued. Then $g(x_1,\ldots,x_k)+x_{k+1}f(x_1,\ldots,x_k)$ is a  permutation polynomial of $i$ variables for every $i\ge k+1$.
	\end{corollary}
	\begin{remark}\label{comparemogr}
		The finiteness condition on $R$ in the last part of Theorem~\ref{constmon} is sufficient  for $\pi_n(\mathcal{MT})_n$ to be a group.  However, if we  consider the case in which $R$ is an infinite field, then $\pi_n(\mathcal{MT}_n)\cong   \mathcal{MT}_n$  because  if   $f\in R_{[k]} $ ($k\ge 1$) maps $R^k$ onto zero then $f$ is necessarily the constant zero polynomial. Therefore, we notice the following  examples:
		\begin{enumerate}
			\item Let $R$ be a real closed field (for example the field of real numbers and the field of real algebraic numbers see~\cite{Vanclosf}). Then $x^{2k+1}$ is a permutation   polynomial on $R$ for every $k\ge 1$ by \cite[Theorem~9]{nolineaper}. So, it is evident that $(x_1^{2k+1},x_2,x_3,\ldots,x_n)$ is a non-invertible element in $\mathcal{MT}_n$ for $k\ge 1$.  Hence,  $\pi_n(\mathcal{MT}_n)$ is not a group since $\pi_n((x_1^{2k+1},x_2,x_3,\ldots,x_n))$ is a non-invertible element in $\pi_n(\mathcal{MT}_n)$ for $k\ge 1$.
			\item Let $R$ be an algebraically closed field of characteristic zero. Then $\pi_n(\mathcal{MT}_n)$ is an infinite group, because  in this case,   $\mathcal{MT}_n$ is a group (we see this later in Proposition~\ref{algcase}).
			\item Let  $R$ be an algebraically closed field of characteristic $p$. Then $ax^{p^t}$ is a permutation polynomial on $R$ for every $a\ne 0$ and $t\ge 1$ by 
			\cite[Theorem~7]{nonlinerpch}. Hence, $(ax_1^{p^t},x_2,x_3,\ldots,x_n)$ is a non-invertible element in $\mathcal{MT}_n$ for every $a\ne 0$ and $t\ge 1$.  Therefore,  $\pi_n(\mathcal{MT}_n)$ is not a group since $\pi_n((ax_1^{p^t},x_2,x_3,\ldots,x_n))$ is a non-invertible element in $\pi_n(\mathcal{MT}_n)$ provided that $a\ne 0$ and $t\ge1$.\label{nonliperm}
		\end{enumerate} 
	\end{remark}
	
	\begin{definition}
		Let $n>1$. We call the monoid $\mathcal{MT}_n$ constructed in Theorem~\ref{constmon}  the monoid of triangular  vector-permutation polynomials or more briefly the triangular monoid.  
	\end{definition}
	\begin{notation}\label{MTSP}
		We set $\mathcal{MT}_1$ to be the  monoid $\mathcal{MP}(R)$ of permutation  polynomials on $R$.
	\end{notation}
	Our next aim is to describe the group of units of the monoid of  vector-permutation polynomials constructed 
	in Theorem~\ref{constmon}. Before  doing so, we need to fix some notation.

	\begin{notation}\label{relations}
		\begin{enumerate}
						\item  Let $(1:f)$ denote $(f(x_1),x_2,\ldots, x_n) $, where $f\in R[x_1]$ is a permutation polynomial on $R$.
			\item For $i>1$, let $(i:u;f)$ denote $ (x_1 ,\ldots, x_{i-1},f +x_iu,x_{i+1},\ldots,x_n   ) $, where $f,u\in R_{[i-1]}$ such that  $u\in UV(R_{[i-1]})$ (i.e. $u$ is a unit-valued polynomial).
			\item  When there is no confusion we write $\Vec{f}\Vec{g}$ instead of $\Vec{f}\circ\Vec{g}$. 
		\end{enumerate}
	\end{notation}
	\begin{remark}\label{deco}
		\begin{enumerate}
			\item 
			
			Let $\Vec{f} \in\mathcal{MT}_n$. By definition,  $\Vec{f}$ has the form ~(\ref{vectel}) in Theorem~\ref{Constgr}. Then in view of Notation~\ref{relations}, we have by Fact~\ref{tricalc},
			\begin{equation*}
				\Vec{f}=(1:f_1)\circ  (2:u_2;f_2)\circ\cdots\circ (n:u_n;f_n), \text{or simply we write}
			\end{equation*}
			\begin{equation}\label{decexp}
				\Vec{f}=(1:f_1)\prod_{i=2}^{n}(i:u_i;f_i),
			\end{equation}
			where  $f_1\in R[x_1]$ is a permutation polynomial  and    $f_i,u_i\in R_{[i-1]}$   such that  $ u_i $  is a unit-valued polynomial for $i=2,\ldots, n$. Clearly, that the expression ~(\ref{decexp}) is uniquely determined by the expression~(\ref{vectel})   
			\item Again in view of Notation~\ref{relations}, we see that for $i>1$,
			$(1:x_1)=(i: 1;0)=(x_1,\ldots,x_n)$ is the identity element of $\mathcal{MT}_n$.
			\item Let $\Vec{g}=(1:g_1)\prod_{i=2}^{n}(i:w_i; g_i)$. To obtain an expression of $\vec{g}\vec{f}$  similar to expression~(\ref{decexp}), we adopt Equations~(\ref{prodre1}) and ~(\ref{prodre2}) to get  
			
			\begin{equation}\label{prodexpres}
				\vec{g}\vec{f}=(1:g_1\circ f_1)\prod_{i=2}^{n}(i:u_i w_i(\vec{v_i});g_i(\vec{v_i})+ w_i(\vec{v_i})f_i), \text{ where }
			\end{equation} 
			\begin{align*}
				\vec{v_i}(x_1,\ldots,x_{i-1})&=(f_1(x_1),
				f_2(x_1)+x_2u_2(x_1),\ldots,
				f_{i-1}(x_1,\ldots,x_{i-1})+x_{{i-1}}u_{i-1}(x_1,\ldots,x_{i-2}));\\
				\quad i=2,\ldots,n.
			\end{align*}
			We notice here that $w_i(\vec{v_i})$ is a unit-valued polynomial since $w_i$ is a unit-valued polynomial. So, expression~(\ref{prodexpres}) can be written in the expression ~(\ref{decexp}) by choosing $h_1=g_1\circ f_1$, 
		\end{enumerate}
	\end{remark}
	The following Lemmas are straightforward.
	\begin{lemma}\label{relat1}
		Let $f,g \in R[x_1]$ be permutation polynomials. Then
		\begin{enumerate}
			\item $(1:f )(1:g)= (1:f\circ g)$;
			\item $\Vec{f} = (1:f)$ is a unit (invertible) in $\mathcal{MT}_n$  if and only if $f$ is invertible (with respect to the composition of polynomials) if and only if $f$ is an $R$-automorphism of $R[x_1]$. In this case
			$\Vec{f}^{-1}=(1:f^{-1})$, where $f^{-1}$ is the inverse of $f$ with respect to composition.
		\end{enumerate} 
	\end{lemma}
	
	\begin{lemma}\label{relat2}
		Let $i>1$ and $f,g,u,v\in R_{[i-1]}$ such that $u$ and $v$ are unit-valued polynomials.
		Then
		\begin{enumerate}
			\item $(i:u;f) (i:v;g)= (i:uv;f+ug)$;
			\item $\Vec{f}=(i\!:u;f)$ is a unit   in $\mathcal{MT}_n$  if and only if $u$ is a unit of the polynomial ring $R_{[i-1]}$ (with respect to multiplication). In this case, $\Vec{f}^{-1}=(i:u^{-1};-u^{-1}f)$, where $u^{-1}$ is the inverse of $u$ in the ring $R_{[i-1]}$ (with respect to multiplication). 		 
		\end{enumerate}
	\end{lemma}
	
	By means   of Notation~\ref{relations}, Remark~\ref{deco}, Definition~\ref{CHSequvfun} and Lemma~\ref{uniqex}, we get   the following useful lemma.
	
	\begin{lemma}\label{equivrela}
		Let $R$ be a finite commutative ring. Let $f_1, g_1\in R[x_1]$ be permutation polynomials on $R$, and 
		$f_i,g_i,v_i,u_i\in R_{[i-1]}$   such that  $ u_i $ and $v_i$ are unit-valued polynomials on $R$    for $i=2,\ldots, n$. The following statements are equivalent
		\begin{enumerate}
			\item $  (1:f_1)\prod_{i=2}^{n}(i:u_i;f_i)\quv (1:g_1)\prod_{i=2}^{n}(i:v_i;g_i)$ on $R^n$;
			\item $(1:f_1)\quv (1:g_1)$ and $(i:u_i;f_i)\quv (i:v_i;g_i)$ on $R^n$ for $i=2,\dots,n$;
			\item $ f_1 \quv g_1$, $ f_i \quv  g_i $ and $ u_i \quv  v_i $ on $R$ for $i=2,\dots,n$.
		\end{enumerate}
	\end{lemma}
	\begin{theorem}\label{Trele}
		Let $R$ be a   commutative ring. Then the group of units of   the monoid of triangular vector-permutation polynomials $\mathcal{MT}_n$ consists  of all elements of the form 	\[\Vec{f}=(1:f_1)\prod_{i=2}^{n}(i:u_i;f_i),  \]
		where  $f_1\in R[x_1]$ is an $R$-automorphism,   $f_i\in R_{[i-1]}$  and $ u_i \in  R_{[i-1]}^\times$   for $i=2,\ldots, n$.
	\end{theorem}
	\begin{proof}
		Let $ \Vec{f}\in \mathcal{MT}_n$. Then, by Remark~\ref{deco},
		$\Vec{f}=(1:f_1)\prod_{i=2}^{n}(i:u_i;f_i)$ where $f_1\in \mathcal{MP}(R)$, $f_i\in R_{[i-1]}$ and $u_i\in UV(R_{[i-1]})$.
		First, suppose that $\vec{f}$ is a unit in $\mathcal{MT}_n$. Then there exists $\vec{g}=(1:g_1)\prod_{i=2}^{n}(i:w_i; g_i)$, where $g_1\in \mathcal{MP}(R)$, $g_i\in R_{[i-1]}$ and $w_i\in UV(R_{[i-1]})$ with $\vec{g}\vec{f}=(x_1,\ldots,x_n)$. By Remark~\ref{deco} (Equation~(\ref{prodexpres})), \begin{equation*} 
			\vec{g}\vec{f}=(1:g_1\circ f_1)\prod_{i=2}^{n}(i:u_i w_i(\vec{v_i});g_i(\vec{v_i})+ w_i(\vec{v_i})f_i), \text{ where } \vec{v_i}\in R_{[i-1]}^{i-1} \text{ for }i=2,\ldots,n.
		\end{equation*}
		Again by Remark~\ref{deco}, we have the following equalities
		$$\begin{matrix}
			g_1\circ f_1=x_1\\
			g_i(\vec{v_i})+ w_i(\vec{v_i})f_ix_i=x_i \text{ for }i=2,\ldots,n.
		\end{matrix}$$
		The first equality implies that $f_1$ is an $R$-automorphism. From the second equality,
		we infer  that $w_i(\vec{v_i})f_i=1$ since $g_i(\vec{v_i}),w_i(\vec{v_i})$ and  $f_i$ are independent from the variable $x_i$ for $i=2,\ldots,n$. Thus, $f_i\in R_{[i-1]}^\times$.
		Conversely, assume that $f_1\in R[x_1]$ is an $R$-automorphism,   $f_i\in R_{[i-1]}$  and $ u_i \in   R_{[i-1]}^\times$   for $i=2,\ldots, n$.
		Then  by construction and Lemma~\ref{relat1}	 and Lemma~\ref{relat2}, $\Vec{f}$ is invertible being
		product (composition) of units in $\mathcal{MT}_n$. 
	\end{proof} 
	From Theorem~\ref{Trele}, Lemma~\ref{relat1} and Lemma~\ref{relat2}  we have the following result.
	\begin{corollary}
		Let $\Vec{f}\in \mathcal{MT}_n$. Suppose that  $\Vec{f}=(1:f_1)\prod_{i=2}^{n}(i:u_i;f_i)$ where $f_0\in\mathcal{MP}(R)$, $f_i,u_i\in  R_{[i-1]}$ such that  $u_i\in UV(R_{[i-1]})$. Then
		$\vec{f}$ is invertible  if and only if $(1:f_1)$ and $(i:u_i;f_i)$ are invertible for $i=2,\ldots,n$. 
		Further,  the inverse is given by the relation   $$ \Vec{f}^{-1}=(\prod_{i=n}^{2}(i:u_i^{-1};-u_i^{-1}f_i))(1:f_1^{-1}),$$
		where $f_1^{-1}$  is the inverse of $f_1$ with respect to composition and for $i=2,\ldots,n$, $u_i^{-1}$ is the inverse of $u_i$ with respect to multiplication. 
	\end{corollary}
	\begin{definition}
		We denote by  $TR_n$  the group of units of the monoid $\mathcal{MT}_n$. We call  $TR_n$ the group of triangular vector-permutation polynomials or shortly the triangular group.
	\end{definition}
		If $R$ is finite, then $\pi_n({\mathcal{MT}_n})$ is a group of permutations of $R^n$ by Theorem~\ref{Trele}. Furthermore, in this case,  $\pi_n({TR_n})$ is also a group of permutations of $R^n$ (being a  closed subset of the finite group $\pi_n({\mathcal{MT}_n})$). An interesting question arises from what we already  mentioned:  Is it true that  $\pi (\mathcal{MT}_n)\ne \pi_n( TR_n)$?
	Fortunately,  we can  generally  answer  this question affirmatively except in the   case in which $R$ is a finite product of the field of two elements $\mathbb{F}_2$. Without loss of generality, we consider only the case that $R$ is a finite local ring since the general case follows by the Chinese Reminder Theorem.
	
	Before doing so,   we recall  the following celebrated fact which is a special case of a more general result due to N\"obauer \cite[Theorem~2.3]{Nopercon}. For a   proof of this fact, we refer the interested reader to a  paper by Nechaev~\cite[Theorem~3]{Necha}. 
	
	Hereafter, for a polynomial $f$, let $f'$ be its formal derivative.
	\begin{lemma}\cite[Theorem.~2.3]{Nopercon}\label{Necha}
		Let $R$ be a finite local ring, that is not a field, $M$ its maximal ideal, and 
		$f\in R[x]$. 
		
		Then $f$ is a permutation polynomial on $ R$ 
		if and only if the following conditions hold:
		\begin{enumerate}
			\item $f$ is a permutation polynomial on $R/M$;
			\item for all $a\in R$, $f'(a)\ne 0\mod{M}$.
		\end{enumerate}
	\end{lemma}
	\begin{lemma}\label{nounitrep}
		Let $R$ be a finite local ring which is not a field with residue field $\mathbb{F}_q$. Then there is no invertible unit-valued polynomial that represents the unit-valued polynomial function induced by $g(x)=(x^q-x)+1$, i.e.  $\pi_1(g)\in \pi_1(UV(R[x])\setminus {R[x]}^\times)$. 
		
	\end{lemma}
	\begin{proof}
		Assume to the contrary that there is an invertible  $f\in R[x]$ such that $f$ and $g$ induce the same function  on $R$. Then by Remark~\ref {uni-v properties}, since $f(0)=g(0)=1$, $f(x)=1 +m_1x+\cdots+m_nx^n$ with  
		$m_1,\ldots,m_n\in M$. Also, $xf(x)$ and  $xg(x)$ induce the same function on $R$.
		However, by Lemma~\ref{Necha}, $xf(x)$ is a permutation polynomial while $xg(x)$ is not (since the derivative of $xg(x)$ admits a root on $\mathbb{F}_q$).  This 
		contradiction implies that such a polynomial $f$ does not exist.
	\end{proof}
	
	\begin{theorem}Let $R$ be a finite local ring. Then
		$\pi_n( TR_n)= \pi_n (\mathcal{MT}_n)$ if and only if $R=\mathbb{F}_2$.
	\end{theorem}
	\begin{proof}
		If $R=\mathbb{F}_2$, one  easily sees that $1$ is the only unit of $R$ and hence every unit-valued polynomial $u(x_1,$\ldots$,x_k)$ in $k$ variables (for any $k\ge1$) is equivalent to the (invertible) constant  polynomial $1$. On the other hand, there are only two  permutations of $R$ that are induced by  the (invertible) permutation polynomials $x$ and $x+1$. 
		
		Now, if $\vec{F} \in \pi_n (\mathcal{MT}_n)$, then $\vec{F}=\pi_n (\vec{f})$ for some $\vec{f}\in \pi (\mathcal{MT}_n)$, where
		\[\Vec{f}=(1:f_1)\prod_{i=2}^{n}(i:u_i;f_i),  \]
		with  $f_1\in \mathcal{MP}(R)$  and    $f_i,u_i\in R_{[i-1]}$   such that  $ u_i\in UV(R_{[i-1]}) $    for $i=2,\ldots, n$. But, the above discussion tells us that $u_i \quv 1$ on $R$ (for $i=2,\ldots, n$) and either $f_1\quv x_1$ or $x_1+1$ on $R$. Without loss of generality  say $f_1\quv x_1$. 
		So, if  we consider $\Vec{g}=(1:x_1)\prod_{i=2}^{n}(i:1;f_i)$, then, by Theorem~\ref{Trele} and Definition~\ref{CHSequvfun}, $\Vec{g}\in TR_n$ and $\pi_n(\Vec{g})=\pi_n(\Vec{f})=\vec{F}$. Thus, $ \vec{F}\in \pi_n(TR_n)$.
		
		For the other direction, we construct a vector-polynomial permutation $\vec{F}\in   \pi_n(\mathcal{MT}_n)\setminus \pi_n(TR_n)$. However, we distinguish two cases on $R$.
		First, assume that $R$ is a local ring with  maximal ideal $M\ne \{0\}$, and  consider the unit-valued polynomial  
		$g(x_1)=(x_1^q-x_1)+1$. Then $\vec{g}=(2:g;0)\in \mathcal{MT}_n\setminus TR_n$ and hence  $\pi (\vec{g})\in \pi (\mathcal{MT}_n)$.
		We claim that $\pi (\Vec{g})\notin \pi ( TR_n)$. Otherwise, suppose that there is a vector-permutation  polynomial $\Vec{f} \in TR_n$,   such  that $\pi_n(\Vec{f})=\pi_n(\Vec{g})$. Then, writing  $\Vec{f}=(1:f_1)\prod_{i=2}^{n}(i:u_i;f_i)$ implies that $u_2$ is a unit of $R[x_1]$ by 
		Theorem~\ref{Trele}.  Since $ \pi (\Vec{g})=\pi_n(\Vec{f})$, we have by Lemma~\ref{equivrela},
		$u_2\quv g$ on $R$. But, such a unit polynomial $u_2$ does not exist by Lemma~\ref{nounitrep}.
		Hence, the claim. 
		
		Now, consider the case $R=\mathbb{F}_q$ with $q\ne 2$. One can  easily see that the unit polynomials are only the non zero-constants. Then, since $q\ne 2$, there is an element $a\in\mathbb{F}_q$ such that $a\ne 0,1$ and we can define a non constant unit-valued function $F\colon \mathbb{F}_q \longrightarrow \mathbb{F}_q$,
		$F(b)=\begin{cases}a & \textnormal{if } b=a  \\
			1 & \textnormal{if } b\ne a 
		\end{cases}.$ Also, by Lagrange interpolation, we can find a polynomial $g$ that represents $F$. Then an argument similar to that of
		the previous case shows that  $\pi_n(2:g;0)\in \pi_n(\mathcal{MT}_n)\setminus \pi_n(TR_n)$. 
				\end{proof}

	\begin{remark}\label{strcembedding}
		Let $k>n$. Then the map
				\[e_{n,k}\colon R_{[n]}^n
		\longrightarrow R_{[k]}^k ,  \vec{f}=(f_1,\ldots,f_n)\mapsto (f_1,\ldots,f_n,x_{n+1}\ldots,x_k)\]  is
		well known to be an embedding and by its  restrictions on  the monoid $\mathcal{MT}_n$ and its group of units $TR_n$ respectively, we have the following embeddings
		\begin{enumerate}
			\item  $\mathcal{MT}_n \hookrightarrow   \mathcal{MT}_k$; and
			\item  $TR_n \hookrightarrow  TR_k$.
		\end{enumerate}
		Also, there is an embedding  $\bar{e}_{n,k}\colon \pi_{n}(\mathcal{MT}_n)
		\longrightarrow \pi_{k}(\mathcal{MT}_k) ,  \pi_{n}( \vec{f}) \mapsto \pi_{k}\circ e_{n,k}( \vec{f})$. 
	\end{remark}
	\section{Some  results on the structure of $\mathcal{MT}_n$}\label{sec4}
	In this section, we  investigate the structures of the monoid  $\mathcal{MT}_n$,   its group of units $TR_n$, and when $ R$ is finite,  its induced group $\pi_n(\mathcal{MT}_n)$ of permutations on $R^n$.
	In particular, we see that the structures of     $TR_n$  and  $\pi_n(\mathcal{MT}_n)$ depend mainly
	on that of $\mathcal{MT}_n$.
	
	From now on, when $H$ is a monoid let $H^\times$ denote its group of units.
	\subsection{The structure  of the monoid $\mathcal{MT}_n$} \par
	In this subsection, we see that the triangular monoid $\mathcal{MT}_n$ 
	can be decomposed into an iterated semi-direct product of $n$ monoids.\par
	The concept of the semi-direct product of monoids
	plays a vital role in this part of the  paper. In the following, we recall its   definition. 
	Let $A$ and $B$ be monoids and $End(B)$ be the monoid of endomorphisms
	of $B$  with respect to composition. If $\phi\colon A\longrightarrow End(B)$, $a\mapsto \phi_a$, is a homomorphism then the semi-direct product
	$B\rtimes_{\phi} A$ (or simply $B\rtimes  A$) is the monoid with elements $\{ (a,b)\colon a\in A, b\in          B\}$ and operation  $(a,b)(c,d)=(ac,b\phi_a(d))$  (when $A$ acts on $B$ from left) or 
	$(a,b)(c,d)=(ac,\phi_c(b)d)$  (when $A$ acts on $B$ from right) (see for example~\cite{monsemid}). We note to the reader that   if we chose $A$ and $B$ to be two groups, and replace $End(B)$ by the group $Aut(B)$ of automorphisms (i.e. $\phi\colon A\longrightarrow Aut(B)$), then we just have the semi-direct product of the group $B$ by $A$ with the same operation as  mentioned  before.  We note here that  $(a,b)^{-1}=(a^{-1},\phi_{a^{-1}}(b^{-1}))$ and $\phi_{a^{-1}}=\phi_{a}^{-1}$.
	\begin{lemma}\label{semdm}
		Fix $2\le k \le n$. Let $\mathcal{ML}^n_{k}$ denote the set of vector-polynomials of the form
		
		$(k:u;f)$,
		where $f,u\in R_{[k-1]}$ such that $u\in UV(R_{[k-1]})$, is a  submonoid of the monoid $\mathcal{MT}_n$ which is isomorphic to the semi-direct product of the additive group of $R_{[k-1]}$ by the monoid $(UV(R_{[k-1]}), ``\cdot")$   of unit-valued polynomials.  That is,
		$\mathcal{ML}^n_{k} \cong R_{[k-1]} \rtimes UV(R_{[k-1]})$.
	\end{lemma}
	\begin{proof}
		It is clear that the identity $(k:1; 0)\in \mathcal{ML}^n_{k}$. Also, by Lemma~\ref{relat2},
		$ \mathcal{ML}^n_{k}$ is closed with respect to composition, whence $\mathcal{ML}^n_{k}$
		is a submonoid of $\mathcal{MT}_n$. 
		Now, given a unit-valued polynomial $u\in  UV(R_{[k-1]})$, one can define
		an endomorphism $\phi_u\colon R_{[k-1]} \longrightarrow R_{[k-1]}$ by 
		$\phi_u(f)=uf$ for every $f\in  R_{[k-1]}$. Consider, the map $\phi\colon UV(R_{[k-1]})\longrightarrow (End( R_{[k-1]}),``\circ")$, ($u \mapsto \phi_u$). We claim that $\phi$ is a homomorphism. 
		To show the claim, let $u,v\in  UV(R_{[k-1]})$ and consider $f\in  R_{[k-1]}$.
		We have,  $$  \phi_u \circ  \phi_v(f)=\phi_u (\phi_v(f))= \phi_u (v f) =uvf=\phi_{uv}(f).$$
		Thus $  \phi_u \circ  \phi_v = \phi_{uv}.$ This proves the claim. Hence, the semi-direct product
		$R_{[k-1]} \rtimes UV(R_{[k-1]})$ is defined with the following  operation
		$$(u,f)(v,g)=(uv,f+ug),$$ where $u,v\in UV(R_{[k-1]})$ and $f,g\in R_{[k-1]}$.
		The rest is to show that the map 
		$$\psi\colon R_{[k-1]} \rtimes UV(R_{[k-1]})\longrightarrow  \mathcal{ML}^n_{k}, \text{ }  (u,f) \mapsto (k:u; f)$$ is an isomorphism. We leave this to the reader. 		
	\end{proof}
	Next, we   show that the monoid $\mathcal{MT}_n$ can be viewed as the semi-direct product  of the  $\mathcal{ML}^n_{n}$ (defined in Lemma~\ref{semdm}) by the monoid $\mathcal{MT}_{n-1}$. 
	We begin with proving the following technical Lemma.
	
	\begin{lemma}\label{Tsemdlma}
		Let $n>1$.
		Then     $      \mathcal{ML}^n_{n}  \rtimes \mathcal{MT}_{ n-1}$ is equipped with the operation
		\begin{equation}\label{semoponTM}
			(\vec{h},((n:u ;f ))(\vec{l},((n:v ;g )) 
			= (\vec{h}\circ\vec{l} ,(n:u(\vec{l})v ;f(\vec{l})+u(\vec{l})g) ),
		\end{equation}where $\vec{h},\vec{l}\in \mathcal{MT}_{ n-1}$ and $(n:u ;f ),(n:v ;g )\in \mathcal{ML}^n_{n}$
	\end{lemma}
	\begin{proof}
		Let $\vec{h}\in \mathcal{MT}_{n-1}$. For simplicity write  $\vec{h}=(h_1, \ldots,h_{n-1})$, where $h_i\in R_{[i]}$ for $i=1,\ldots,n-1$ (keeping in mind that $\vec{h}$ satisfies the conditions of expression~(\ref{vectel}) in Theorem~\ref{Constgr}). Then define a map  $\phi_{\vec{h}}\colon \mathcal{ML}^n_{n} \longrightarrow \mathcal{ML}^n_{n}$ by 
		$((n:u;f))\phi_{\vec{h}}=(n:u(\vec{h});f(\vec{h}))$  for   $(n:u ;f )\in \mathcal{ML}^n_{n}$. Since  $u$ is a unit-valued polynomial (in $n-1$ variables), we have  that  $u(\vec{h})$ is a unit-valued polynomial (in $n-1$ variables). This shows $((n:u;f))\phi_{\vec{h}}\in \mathcal{ML}^n_{n}$. Now, let $(n:v ;g )\in \mathcal{ML}^n_{n}$ and consider
		\begin{align*}
			((n:u;f)(n:v;g))\phi_{\vec{h}} &=((n:uv;f+ug))\phi_{\vec{h}} \text{ by Lemma~\ref{relat2}}\\
			& =(n:(uv)(\vec{h});f(\vec{h})+(ug)(\vec{h}))=(n: u(\vec{h})v(\vec{h});f(\vec{h})+ u(\vec{h})g(\vec{h}))\\
			&\displaystyle_{=}^{\text{Lemma~\ref{relat2}}} (n: u(\vec{h});f(\vec{h}))(n: v(\vec{h}); g(\vec{h}))=
			((n:u;f))	\phi_{\vec{h}}((n:v;g))	\phi_{\vec{h}}.
		\end{align*}
		Thus, $\phi_{\vec{h}}\in End(\mathcal{ML}^n_{n})$. Then,  the map $\phi\colon \mathcal{MT}_{ n-1}\longrightarrow End(\mathcal{ML}^n_{n})$, ${\vec{h}}\mapsto \phi_{\vec{h}}$ is a homomorphism. Indeed, if $\vec{l}=(l_1,\ldots,l_{n-1})\in \mathcal{MT}_{n-1}$, we have
		\begin{align*}
			((n:u ;f ))\phi_{\vec{h}\circ\vec{l}} &=(n:u(\vec{h}\circ\vec{l}) ;f(\vec{h}\circ\vec{l}) )
			=(n:u(h_1(\vec{l}),\ldots,h_{n-1}(\vec{l})) ;f(h_1(\vec{l}),\ldots,h_{n-1}(\vec{l})) )\\
			& =(n:u(h_1,\ldots,h_{n-1}) ;f(h_1 ,\ldots,h_{n-1}) )\phi_{\vec{l}}=((n:u ;f ))\phi_{\vec{h}})\phi_{\vec{l}}\\
			&=((n:u ;f ))\phi_{\vec{h}}\circ\phi_{\vec{l}}.
				\end{align*}
		Thus, the semi-direct product 
		$\mathcal{ML}^n_{n} \rtimes \mathcal{MT}_{ n-1}$ is defined with the following  operation
		\begin{align*}
			(\vec{h},((n:u ;f ))(\vec{l},((n:v ;g )) &=(\vec{h}\circ\vec{l} ,((n:u ;f ))\phi_{\vec{l}}(n:v ;g ))\\
			&= (\vec{h}\circ\vec{l} ,(n:u(\vec{l})v ;f(\vec{l})+u(\vec{l})g) )
		\end{align*}
		or more explicitly
		\begin{align}\label{explcform}(\vec{h},(n:u ;f ))(\vec{l},(n:v ;g ))&= 
			((h_1\circ l_1,h_2(l_1,l_2),\ldots,h_{n-1}(l_1,\ldots,l_{n-1})) ,(n:u(\vec{l})v ;f(\vec{l})+u(\vec{l})g)).\end{align}\qedhere
	\end{proof}
	\begin{proposition}\label{Tsemd}
		Let $n>1$ and let $\mathcal{ML}^n_{n}$ denote the monoid of vector-polynomials of the form
		
		$(n:u;f)$,
		where $f,u\in R_{[k-1]}$ and    $u\in UV(R_{[k-1]})$. Then     $\mathcal{MT}_n \cong      \mathcal{ML}^n_{n}  \rtimes \mathcal{MT}_{ n-1}$.
		
	\end{proposition}
	\begin{proof}
		In view of Lemma~\ref{Tsemdlma}, the semi-direct product $ \mathcal{ML}^n_{n}  \rtimes \mathcal{MT}_{ n-1}$ is defined by  the operation given in Equation~(\ref{semoponTM}) (more explicit in Equation~(\ref{explcform})).
		Define  a map $\psi\colon \mathcal{ML}^n_{n}  \rtimes \mathcal{MT}_{n-1} \longrightarrow  \mathcal{MT}_{ n}   $
		as the following: if $(\vec{h}, (n: u; f) )\in \mathcal{ML}^n_{n}  \rtimes \mathcal{MT}_{n-1}$, where  $u\in UV(R_{[n-1]})$  and $\vec{h}=(h_1, \ldots,h_{n-1})$; $h_i\in R_{[i]}$ for $i=1,\ldots,n-1$, then set $$\psi (\vec{h}, (n: u; f) )=(h_1,\ldots, h_{n-1}, f+ux_n).$$   Note that, by the definition of $\mathcal{MT}_{n-1}$, $h_1$ is a permutation polynomial on $R$ and $h_i= f_i+ u_ix_i$ where  $f_i,  \in R_{[i-1]}$ and $u_i\in UV(R_{[i-1]})$   for $i=2,\ldots, n-1$. This means that $\psi (\vec{h}, (n: u; f) )\in  \mathcal{MT}_{n}$. Now, we show that $\psi$ is a homomorphism. Let $(\vec{h}, (n: u; f) )$ and  $(\vec{l}, (n: v; g) )\in \mathcal{ML}^n_{n}  \rtimes \mathcal{MT}_{n-1}$. Then we have  
		\begin{align*} \psi( (\vec{h},(n:u ;f )&)(\vec{l},(n:v ;g )))\\ &\displaystyle_{=}^{\text{Equation~\ref{explcform}}} 
			\psi(((h_1\circ l_1,h_2(l_1,l_2),\ldots,h_{n-1}(l_1,\ldots,l_{n-1})) ,(n:u(\vec{l})v ;f(\vec{l})+u(\vec{l})g)))\\
			&	=(h_1\circ l_1,h_2(l_1,l_2),\ldots,h_{n-1}(l_1,\ldots,l_{n-1}) ,f(\vec{l})+u(\vec{l})g+u\big(\small{\Vec{l}}\big)v x_n)\\
			&	=(h_1\circ l_1,h_2(l_1,l_2),\ldots,h_{n-1}(l_1,\ldots,l_{n-1}) ,f(l_1,\ldots,l_{n-1})+u(l_1,\ldots,l_{n-1})g\\
			& \quad + u(l_1,\ldots,l_{n-1})vx_n)\\
			& =(h_1,\ldots,h_{n-1},f+ux_n)(l_1,\ldots,l_{n-1},g+vx_n)\\
			& = \psi( (\vec{h},(n:u ;f )))\psi((\vec{l},(n:v ;g ))).\end{align*}
		The hardest part of the proof has been finished,   and the remaining details are straightforward and left to the reader.
		\end{proof}
	
	\begin{convention}
		From now on,	 by $A_1 \rtimes A_2 \rtimes  A_3$, we mean $A_1 \rtimes ( A_2 \rtimes  A_3)$, i.e. we start computing from the left.
	\end{convention}
	We are now in position  to decompose the triangular monoid  $\mathcal{MT}_{n}$ into an iterated semi-direct product of $\mathcal{ML}^n_{n},\ldots, \mathcal{ML}^2_{2}$, and  $\mathcal{MT}_1=\mathcal{MP}(R)$ (see Notation~\ref{MTSP}). In fact, we  deduce more.
	\begin{theorem}\label{mondec}
		Let $n>1$. Then \begin{enumerate}
			
			\item 	$\mathcal{MT}_n  \cong  \mathcal{ML}^n_{n} \rtimes \cdots  \rtimes \mathcal{ML}^2_{2}\rtimes  \mathcal{MP}(R)$;

			\item $	\mathcal{MT}_n \cong   ( R_{[n-1]}]  \rtimes UV(R_{[n-1]}))  \rtimes \cdots \rtimes ( R {[x_1]}  \rtimes UV(R{[x_1]}))\rtimes \mathcal{MP}(R)$.
		\end{enumerate}
		
	\end{theorem}
	\begin{proof}
		By Proposition~\ref{Tsemd},   $\mathcal{MT}_n \cong      \mathcal{ML}^n_{n}  \rtimes \mathcal{MT}_{ n-1} $.  Then iterated applications of this relation yield the first expression. Then, we use   Lemma~\ref{semdm}, to  have $\mathcal{ML}^i_{i} \cong R_{[i-1]} \rtimes UV(R_{[i-1]})$ for $i=2,\ldots,n$.  Combining  this with the first expression  gives the second expression. 
			\end{proof}
	\subsection{The structure of the triangular group $TR_n$}
	\par 
	
	Using the fact that monoid  homomorphisms preserve the identity, we  will see that the group units  of a semi-direct product of monoids is just the semi-direct product of their groups of units. Then, we apply this result to the decomposition of $\mathcal{MT}_n$ given in Theorem~\ref{mondec} to obtain 
	a decomposition  of  the  group $TR_n$ into an iterated semi-direct product of  groups.
	\begin{lemma}\label{isomon}
		Let $A$ and $B$ be monoids. Consider the homomorphism $ \phi\colon A\longrightarrow End(B)$ ($a\mapsto \phi_a$, $\phi_a\colon B \longrightarrow B$ is a homomorphism), and let $\psi_a$ be the restriction of 
		$\phi_a$ onto $B^\times$ then
		\begin{enumerate}
			\item $\phi_a\in Aut(B)$ for every $a\in A^\times$; \label{1is}
			\item $\psi_a\in Aut(B^\times)$ for every $a\in A^\times$. \label{2is}
		\end{enumerate}
		
	\end{lemma}
	\begin{proof}
		Consider an element $a\in A^\times$. By the definition of $\phi$, $\phi_a$ is a homomorphism. So  we  only need to show that $\phi_a$ is invertible. Since a homomorphism  preserves the identity, $\phi(1_A)=\phi_{1_A}$ is the identity endomorphism of $B$. Hence, by the properties of $\phi$,
		$$\phi_a\circ \phi_{a^{-1}}=\phi(a)\circ \phi({a^{-1}})=\phi(aa^{-1})=\phi(1_A)=\phi(a^{-1}a)=\phi_{a^{-1}}\circ
		\phi_a. $$
	Thus, $\phi_a$ is invertible, and 	therefore $ \phi_a \in Aut(B)$. This proves~(\ref{1is}).
		
		(\ref{2is})  Since a homomorphism  preserves units and the restriction of a homomorphism is a homomorphism, 
		we have  that $\psi_a\colon B^\times \longrightarrow B^\times$   is a homomorphism for every $a\in A^\times$. Then an  argument similar to that
		of (\ref{1is}) shows that $\psi_a\in Aut(B^\times)$ for every $a\in A^\times$.
			\end{proof}
	\begin{remark}\label{restrictmonho}
		If 
		$A$ and $B$ are monoids and $ \phi\colon A\longrightarrow End(B)$ ($a\mapsto \phi_a$, $\phi_a\colon B \longrightarrow B$ is a homomorphism)  is a homomorphism, then   we have a homomorphism   $ \psi\colon A^\times\longrightarrow Aut(B^\times)$ ($a\mapsto \psi_a$, where $\psi_a$ is the restriction of 
		$\phi_a$ onto $B^\times$) by Lemma~\ref{isomon}.  Furthermore, the semi-direct product (of groups)  $B^\times \rtimes_{\psi} A^\times$  with respect to   $\psi$ is defined with operation inherited from that of the semi-direct product (of monoids) $B  \rtimes_{\phi} A$. 
	\end{remark}
	
	\begin{proposition}\label{unisemi}
		
		Let $A$ and $B$ be monoids and $ \phi\colon A\longrightarrow End(B)$ ($a\mapsto \phi_a$, $\phi_a\colon B \longrightarrow B$ is a homomorphism)  be a homomorphism. Let    $ \psi\colon A^\times\longrightarrow Aut(B^\times)$ be the homomorphism   defined by  $a\mapsto \psi_a$, where $\psi_a$ is the restriction of $\phi_a$ onto $A^\times$, then 
		$$(B\rtimes_{\phi} A)^\times =B^\times \rtimes_{\psi} A^\times.$$

	\end{proposition} 
	\begin{proof}
		By the definition of $\psi$, the operation of $B^\times \rtimes_{\psi}A^\times $ is just the operation of $B  \rtimes_{\phi}A $ restricted  into the set $A^\times \bigtimes B^\times$. Thus, 
		$B^\times \rtimes_{\psi}A^\times \subseteq (B  \rtimes_{\phi}A)^\times$,   and hence  	$B^\times \rtimes_{\psi}A^\times$ is a subgroup of $(B  \rtimes_{\phi}A)^\times$. For the other  inclusion, consider  $(a,b)\in (B  \rtimes_{\phi}A)^\times$. Then there exists $(a_1,b_1)\in (B  \rtimes_{\phi}A)^\times$ such that 
		\[(a,b)(a_1,b_1)=(aa_1,b\phi_a(b_1))=(1_A,1_B)=(a_1a,b_1\phi_{a_1}(b))=(a_1,b_1)(a,b).\] 
		Therefore, we have the following two equalities 
		$$aa_1=a_1a =1_A, \text{ and }$$
		$$	b\phi_a(b_1)=b_1\phi_{a_1}(b)=1_B.$$
		The first equality implies that $a\in A^\times$ and $a_1=a^{-1}$, while the other implies 
		$b$ is right invertible.  Also, applying the automorphism  $\phi_a$ to the equality $b_1\phi_{a_1}(b)=1_B$ yields
		since $a_1=a^{-1}$,
		$$\phi_a(b_1)b=\phi_a(b_1)\phi_{1_A}(b)=\phi_a(b_1)\phi_{aa^{-1}}(b)=\phi_a(b_1)\phi_a( \phi_{a_1}  (b))=\phi_a(b_1\phi_{a_1}(b))=\phi_a(1_B)=1_B,$$
		hence $\phi_a(b_1)b=1_B$, and $b$ is a left invertible. Thus, $b\in B^\times$ (being left and right invertible). Therefore, $(a,b)\in B^\times \rtimes_{\psi} A^\times$.
		\end{proof}
	
	\begin{proposition}\label{semiprod}
		Fix $2\le k \le n$. Let $\mathcal{L}^n_{k}$ denote the set of vector-polynomials of the form
				$(k:u;f)$,
		where $f \in R_{[k-1]}$ and $u\in R_{[k-1]}^\times$. Then
		\begin{enumerate}
			\item $\mathcal{L}^n_{k}$	is a  subgroup of the group $TR_n$;\label{1}
			\item $\mathcal{L}^n_{k} \cong R_{[k-1]} \rtimes R_{[k-1]}^\times$, that is, $\mathcal{L}^n_{k}$ is  the semi-direct product of additive group of the polynomial ring in $k-1$ variables by its group of units;\label{2}
			\item $\mathcal{L}^n_{k}\lhd TR_n$ if and only if $n=k$. In this case, $$	TR_n \cong      \mathcal{L}^n_{n}  \rtimes TR_{ n-1}. \label{3}$$
		\end{enumerate}
			\end{proposition}
	\begin{proof}
		(\ref{1})  Let $\mathcal{ML}^n_{k}$ be the monoid  defined in Lemma~\ref{semdm} containing all  vector-polynomials of the form
		$(k:u;f)$,
		where $f \in R_{[k-1]}$ and $u\in UV(R_{[k-1]})$.
		Then, by Theorem~\ref{Trele} $${\mathcal{ML}^n_{k}}^\times= \mathcal{ML}^n_{k}\cap \mathcal{MT}_n^\times=\mathcal{ML}^n_{k}\cap TR_n=\mathcal{L}^n_{k}.$$
		(\ref{2})
		By Lemma~\ref{semdm}, $\mathcal{ML}^n_{k} \cong R_{[k-1]} \rtimes UV(R_{[k-1]})$ (semi-direct product of monoids). But, by definition,  $UV(R_{[k-1]})^\times=R_{[k-1]}^\times$; and $ (R_{[k-1]},+)^\times=R_{[k-1]}$.
		Thus, the semi-direct product 
		$R_{[k-1]} \rtimes R_{[k-1]} ^\times$ is defined by Remark~\ref{restrictmonho}.
		Now, $$\mathcal{L}^n_{k}\displaystyle_{=}^{(\ref{1})}{\mathcal{ML}^n_{k}}^\times \displaystyle_{\cong}^{Lemma~\ref{semdm}}  (R_{[k-1]} \rtimes UV(R_{[k-1]}))^\times\displaystyle_{=}^{Proposition~\ref{unisemi}} 
		R_{[k-1]} \rtimes R_{[k-1]}^\times.$$
		
		(\ref{3}) Assume  that $k=n$. Then,  by Theorem~\ref{constmon}, Remark~\ref{restrictmonho} and Proposition~\ref{unisemi}, we have $$TR_n=\mathcal{MT}_n^\times \cong      (\mathcal{ML}^n_{n}  \rtimes \mathcal{MT}_{ n-1})\!^\times=  \mathcal{L}^n_{n}  \rtimes TR_{ n-1}.$$
		This also, shows that $\mathcal{L}^n_{n}\lhd TR_n$. 
		Now, assume that  $2 \le k<n$ and let $(k :1;x_{k-1})\in \mathcal{L}^n_{k}$. 	So, if  $(k+1:1;x_k)\in TR_{n}\setminus \mathcal{L}^n_{k}$, 	$(k+1:1;x_k)^{-1}= (k+1:1;-x_k)$ by Lemma~\ref{relat2}. 
		Then,  we have by direct calculations (Definition~\ref{compose}), 
		\begin{align*}
			(k+1:1;x_k)(k :1;x_{k-1})(k+1:1;-x_k) &	= (x_1,\ldots,x_{k-1},x_{k-1}+x_{k},x_{k-1}+x_{k+1},x_{k+2},\ldots,x_n) \\
			& =(k :1;x_{k-1})(k+1:1; x_{k-1})  \notin \mathcal{L}^n_{k}.\qedhere
		\end{align*} 
			\end{proof}
	As a result of Proposition~\ref{semiprod}, we are able to decompose the group $TR_n$ into an iterated semi-direct product of $\mathcal{L}^n_{n},\ldots, \mathcal{L}^2_{2}$, and  $TR_1=Aut_R(R[x])$. Alternatively, we can replace  $\mathcal{L}^k_{k} $ by $ R_{[k-1]}]  \rtimes R_{[k-1]}^\times$ ($k=2,\ldots, n$) to obtain another equivalent decomposition. 
	
	\begin{theorem}\label{decoftg}
		Let $n>1$ and let $TR_n$ be the triangular group. Then \begin{enumerate}
			
			\item 	$TR_n  \cong  \mathcal{L}^n_{n} \rtimes \cdots  \rtimes \mathcal{L}^2_{2}\rtimes  Aut_R(R[x])$;

			\item $	TR_n \cong   ( R_{[n-1]}  \rtimes R_{[n-1]}^\times)  \rtimes \cdots \rtimes ( R {[x_1]}  \rtimes R{[x_1]}^\times)\rtimes Aut_R(R[x])$.
		\end{enumerate}
		Moreover, $TR_n$ is solvable if and only if $Aut_R(R[x])$ is solvable. In particular, if $R$ is a domain, $TR_n$ is solvable, and
		$	TR_n \cong   ( R_{[n-1]}]  \rtimes R^\times)  \rtimes \cdots \rtimes ( R {[x_1]}  \rtimes R^\times)\rtimes  ( R    \rtimes R^\times)$.
		
	\end{theorem}
	\begin{proof}
		By Proposition~\ref{semiprod},   $	TR_n \cong      \mathcal{L}^n_{n}  \rtimes TR_{ n-1} $.  Then iterated application of this relation yields the first expression.  Again, by   Proposition~\ref{semiprod}, we have $\mathcal{L}^i_{i} \cong R_{[i-1]} \rtimes R_{[i-1]}^\times$ for $i=2,\ldots,n$.  Combining these results yields the second expression. Moreover, it is well known that Abelian groups are solvable and extensions
		of solvable groups are solvable. Therefore, by the second expression, $TR_n$ is solvable if and only if $Aut_R(R[x])$ is solvable.  In particular, if $R$ is a domain, then 
		$Aut_R(R[x])\cong  R  \rtimes R ^\times$ and for each $i$, $R_{[i]}^\times=R^\times$. Thus, $TR_n$ is solvable, and  the claim follows.
	\end{proof}
	\begin{remark}\label{coincidestrctures}
		\begin{enumerate}

			\item In algebraic geometry, the Jonqui\`eres (classical) triangular group $KTR_n$ is defined to be the group of all vector-polynomials of the form  \begin{equation}\label{trieq}
				\Vec{f}=(
				a_1x_1+b_0,
				f_1(x_1)+a_2x_2,
				\ldots,
				f_{n-1}(x_1,\ldots,x_{n-1})+a_nx_n 
				),
			\end{equation}

			where $a_1,\ldots,a_n\in R^\times$ and $b\in R$  (see \cite{Autpoly1}). In general, we have the inclusions: \begin{equation}\label{eqstrctures}
				KTR_n\subseteq TR_n \subseteq\mathcal{MT}_n.
			\end{equation} 
			However,  by Theorem~\ref{decoftg} and Theorem~\ref{mondec}, one can  easily  see that 
			$$KTR_n= TR_n =\mathcal{MT}_n $$  
			
			if and only if:		\begin{enumerate}
				\item $UV(R_{[i]})=R_{[i]}^\times=R^\times$.\label{coa}
				
				\item $ \mathcal{MP}(R)=Aut_R(R[x])=\{ax+b \colon a\in R^\times \}$.\label{cob}
			\end{enumerate}

			In the following  examples, we demonstrate how the  validity of  the previous two conditions determines the nature  of the inclusions of relation~(\ref{eqstrctures}).
						
			\begin{itemize}
				\item When $R$ is a $D$-ring (see for example \cite{Drings}), $R$ satisfies condition (\ref{coa}) by the definition of $D$-rings. Then
				$KTR_n= TR_n$ (note that it is enough here to have $R_{[i]}^\times=R^\times$). 
					\item If $R$ is a non-algebraically closed field  (for example, real numbers and finite fields),   $$KTR_n= TR_n \subsetneq\mathcal{MT}_n.$$
				In this case, it is clear that   $R_{[i]}^\times=R^\times$  (for each $i$) and  $Aut_R(R[x])$ consists of all linear polynomials.
				Also, since $R$ is not an algebraically closed field, there is an irreducible polynomial $f$ of degree $\ge 2.$ Evidently, $f\in UV(R_{[i]})\setminus  R_{[i]}^\times$.
				\item When $R$ is not reduced (i.e. containing a non-zero nilpotent), we have 
				$$KTR_n\subsetneq TR_n \subseteq \mathcal{MT}_n.$$  Here, we have by Remark~\ref{uni-v properties}, $1+rx_1 \in R_{[i]} ^\times\ne R^\times$ for any non-zero nilpotent $r$. In particular, if $R$ is a finite local ring that is not a field, 
				$$KTR_n\subsetneq TR_n \subsetneq \mathcal{MT}_n,$$  since in this case we also have 
				$1+(x^q-x)\in UV(R[x])\setminus R[x]^\times$.
			\end{itemize}
			\item In Equation~(\ref{trieq}), if $a_1=\cdots=a_n=1$, we get the so called unitriangular group of automorphisms.
			In~\cite{Trgroup}, a decomposition of the unitriangular group on fields  of characteristic zero  into an iterated semi-product of Abelian groups is given.
		\end{enumerate}	
	\end{remark}
	For algebraically closed fields we have the following result.
	
	\begin{proposition}\label{algcase}
		Let $R$ be an algebraically closed field.  
		\begin{enumerate} 
			\item If $R$ of characteristic zero, then   
			$$KTR_n= TR_n =\mathcal{MT}_n.$$
			\item If $R$ of characteristic $p\ne 0$, then
			$$KTR_n= TR_n \subsetneq\mathcal{MT}_n.$$
		\end{enumerate}
	\end{proposition}
	\begin{proof}
		First assume that $R$ is any algebraically closed field. Then  $UV(R_{[i]})=R_{[i]}^\times=R^\times$ since every
		non-constant polynomial has a root over an algebraically closed field. Also, since $R$ is a field,  $Aut_R[x]=\{ax+b \colon a\in R^\times \}$. Thus, $KTR_n= TR_n $, by Remark~\ref{coincidestrctures}.
		
		Now, assume that $R$ is an  algebraically closed field of characteristic zero. Then $ \mathcal{MP}(R)=Aut_R(R[x])$ since $R$ has no non-linear permutation polynomial by  \cite[Theororem~6]{nolineaper}. This together with the first part of the proof implies    $TR_n =\mathcal{MT}_n $ by Remark~\ref{coincidestrctures}. This proves (1).
		
		For the  case where  $R$ is an  algebraically closed field of characteristic $p\ne 0$, we have that $x^p\in  \mathcal{MP}(R)\setminus Aut_R(R[x])$  by Remark~~\ref{comparemogr} (\ref{nonliperm}).
		Thus, $TR_n \subsetneq\mathcal{MT}_n$ by Remark~\ref{coincidestrctures}.
			\end{proof}
	\begin{remark}
		Unlike  solvability,  the decomposition of Theorem~\ref{decoftg} can not determine whether the group
		$TR_n$ is nilpotent or not. However, in the following we notice  for a large class of rings that $TR_n$ is not nilpotent.		 
		\begin{enumerate}
			\item  Let $n>1$ and $R$ is an infinite field. Then $TR_n$ is not nilpotent.
			\item Let $n>2$ and $R$ be  a commutative ring with 		
			Jacobson radical $J(R)$  such that $R/J(R)$ is the finite field of $q$ elements $\mathbb{F}_q$.
			\end{enumerate}	
			To show these assertions, it is enough to find a subgroup of $TR_n$, which is not nilpotent. 
			
			Indeed the first assertion follows from the fact that the unitriangular group (defined in  ~Remark~\ref{coincidestrctures}) is not nilpotent  (see~\cite[Theorem 2.1]{Trgroup}).
			To prove the second assertion, we consider first the case $R=\mathbb{F}_q$. But, then the unitrianguar group is not nilpotent (see \cite[Theorem 3]{Leshchenkounitrfinitef}).
			For the case $J(R)\ne\{0\}$, we note that there exists a natural epimorphism $\pi \colon R \longrightarrow  R/J(R)$ given by $\pi(a) = a \mod J(R)$. Evidently $\pi$ can be extended in  to an epimorphism of monoids $\tilde{\pi} \colon R_{[n]}^n \longrightarrow  ({R/J(R)})_{[n]}^n$, and 
			one easily see that $\tilde{\pi}(UT_n(R))= UT_n(R/J(R))=UT_n(\mathbb{F}_q) $, where $UT_n(A)$ stands
			for the unitriangular group over the ring $A$. Hence $UT_n(R)$ is not nilpotent since $\tilde{\pi}(UT_n(R))$ is not nilpotent. Thus, $TR_n$ is not nilpotent.
	
			\end{remark}
	\subsection{The structure of the induced permutations group $\pi_n(\mathcal{MT}_n)$}
	Recall that $\pi_n$ stands for the natural homomorphism that  assigns for each $\vec{f}\in   R_{[n]} ^n$ its induced vector-function $\vec{F}\colon R^n \longrightarrow R^n$.
	
	In this subsection, we explore  the structure of the group $\pi_n(\mathcal{MT}_n)$ of permutations of $R^n$   induced by the elements of the triangular monoid $\mathcal{MT}_n$. We see that this group can be decomposed into an iterated semi-direct product of $n$ groups in a similar manner to what we have  already seen for the group $TR_n$.
	
	\begin{proposition} \label{imsemiprod}
		Let $R$ be a finite commutative ring and fix $2\le k \le n$. Let $\mathcal{ML}^n_{k}$ be the monoid of vector-permutation polynomials of the  form 	$(k:u;f)$,
		where $f \in R_{[k-1]}$ and $u\in UV(R_{[k-1]})$. Then
		\begin{enumerate}
			\item $\pi_n(\mathcal{ML}^n_{k})$	is a  subgroup of the group $\pi_n(\mathcal{MT}_n)$;\label{p1}
			\item $\pi_n(\mathcal{ML}^n_{k})\cong \PolFun[R^{k-1}] \rtimes \UPF[R^{k-1}] $, that is, $\mathcal{ML}^n_{k}$ is  the semi-direct product of the additive group of the ring of polynomial functions  (on $R$) in $k-1$ variables by its group of units;\label{p2}
			\item $\pi_n(\mathcal{ML}^n_{k}) \lhd \pi_n(\mathcal{MT}_n)$ if and only if $n=k$. In this case, $$	 \pi_n(\mathcal{MT}_n)\cong      \pi_n(\mathcal{ML}^n_{n})  \rtimes  \pi_n(\mathcal{MT}_{n-1}). \label{p3}$$
		\end{enumerate}			
	\end{proposition}
	\begin{proof}
		Let $k$ be fixed.
		(\ref{p1}) Obvious. 
		
		(\ref{p2}) By Lemma~\ref{semdm}, $\mathcal{ML}^n_{k} \cong R_{[k-1]} \rtimes UV(R_{[k-1]})$.
		Then it is easy to see that the sets	
		$A=\{(k:1;f): f\in R_{[k-1]}\}$ and $B=\{(k:u;0): u\in UV(R_{[k-1]}\}$ are isomorphic to $R_{[k-1]}$ and $UV(R_{k-1})$ respectively.  
		Further, 		$\pi_n(A)$ and  $\pi_n(B)$ are subgroups of the finite group $\pi_n(\mathcal{ML}^n_{k})$ (being non empty and closed with respect to composition).

		Now, we claim that $\pi_n(A)\cong \PolFun[R^{k-1}] $ and  $\pi_n(B)\cong \UPF[R^{k-1}]$. To show  this claim, 
		let $\pi\colon R_{[k-1]}\longrightarrow \PolFun[R^{k-1}]$ be the natural  epimorphism that  maps every polynomial to its induced function on $R$. Then define a map $\Psi\colon \pi_n(A)\longrightarrow \PolFun[R^{k-1}]$ by $\Psi(\vec{F})=\pi(f)$, where $f\in R_{[k-1]}$ such that $\vec{F}=\pi_n((k:1;f))$.
		By Lemmas~\ref{uniqex} and~\ref{equivrela}, $\Psi$ is well defined and injective. Let $\vec{G}\in \pi_n(A)$. Then $\vec{G}=\pi_n((k:1;g))$ for some $g\in R_{[k-1]}$.
		Consider, \begin{align*}
			\Psi (\vec{F}\vec{G}) &=\Psi(\pi_n((k:1;f))\circ\pi_n((k:1;g)))= \Psi(\pi_n((k:1;f)(k:1;g)))\\
			&\displaystyle_{=}^{Lemma~\ref{relat2}}\Psi(\pi_n((k:1;f+g))=\pi(f+g)\\
			&=\pi(f)+\pi(g)=\Psi (\vec{F})+\Psi(\vec{G}). 
		\end{align*}
	This shows that $\Psi$ is a homomorphism. 
		By the definitions of the set $A$, and the maps  $\Psi$, $\pi_n$ and $\pi$, we have	
		$$\Psi(\pi_n(A))=\pi(R_{[k-1]})=\PolFun[R^{k-1}],$$ whence   $\Psi$ is surjective.
		 		Therefore, $\pi_n(A)\cong \PolFun[R^{k-1}]$.
			Similarly, one can show that $\pi_n(B)\cong \UPF[R^{k-1}]$.
		
		Now consider   the map $\Phi\colon \pi_n(\mathcal{ML}^n_{k})\longrightarrow \pi_n(B), 
		(\pi_n((k:u;f))\mapsto\pi_n((k:u;0)))$.  It is  obvious that $\Phi$ is surjective. 
		So, if $\pi_n((k:u;f)),\pi_n((k:v;g))\in \pi_n(\mathcal{ML}^n_{k})$,  we have
		\begin{align*}
			\Phi(\pi_n((k:u;f))\circ\pi_n((k:v;g))) & = \Psi(\pi_n((k:u;f)(k:v;g)))\\
			& =\Phi(\pi_n((k:uv;f+ug))=(\pi_n((k:uv;0))\\
			& =\pi_n((k:u;0)(k:v;0)) =\pi_n((k:u;0)\circ \pi_n((k:v;0))\\
			& = \Phi(\pi_n((k:u;f)))\circ\Phi(\pi_n((k:v;g))). 
		\end{align*}
	Thus, $\Phi$ is an epimorphism.
		Now, if $j \colon \pi_n(B) \longrightarrow \pi_n(\mathcal{ML}^n_{k})$ is the inclusion,
		then evidently, $\Phi j=id_{\pi_n(B)}$. Therefore, the following extension splits
		$$1 \rightarrow \pi_n(A) \xrightarrow{i} \pi_n(\mathcal{ML}^n_{k}) \xrightarrow{\Phi} \pi_n(B) \rightarrow 1.$$
		Thus, $$\pi_n(\mathcal{ML}^n_{k})=\pi_n(A)\rtimes\pi_n(B)\cong \PolFun[R^{k-1}] \rtimes \UPF[R^{k-1}].$$
		(\ref{p3})
		Let $\vec{F}\in  \pi_n(\mathcal{MT}_n)$. Then $\vec{F}=\pi_n(\vec{f})$ for some $\vec{f}=(f_1,\ldots,f_{n-1},f_n)$ with $f_i\in R_{[i]}$ for $i=1,\ldots,n$ (by the definition of the triangular monoid $\mathcal{MT}_n$). 
		Define  $P \colon \pi_n(\mathcal{MT}_n) \longrightarrow \pi_{n-1}(\mathcal{MT}_{n-1})$ by
		letting $P(\vec{F})= \pi_{n-1}((f_1,\ldots,f_{n-1}))$. Clearly, $P$ is an epimorphism. Also, 
		the map  $j \colon \pi_{n-1}(\mathcal{MT}_{n-1}) \longrightarrow \pi_n(\mathcal{MT}_n)$, $(\pi_{n-1}((f_1,\ldots,f_{n-1})))\mapsto \pi_{n}((f_1,\ldots,f_{n-1},x_n))$
		is obviously a monomorphism. Further, $Pj=id_{\pi_{n-1}(\mathcal{MT}_{n-1})}$. Hence,
		the following extension splits 
		$$1 \rightarrow \pi_n(\mathcal{ML}^n_{n})\xrightarrow{i} \pi_n(\mathcal{MT}_n) \xrightarrow{P} \pi_{n-1}(\mathcal{MT}_{n-1}) \rightarrow 1.$$
		Therefore,  $	 \pi_n(\mathcal{MT}_n)\cong      \pi_n(\mathcal{ML}^n_{k})  \rtimes  \pi_{n-1}(\mathcal{MT}_{n-1}), \text{ and }  \pi_n(\mathcal{ML}^n_{k})\lhd\pi_n(\mathcal{MT}_n).$
		
		Now, assume that $2\le k<n$. Then, we adopt the argument given in the last part of the proof of Proposition~\ref{semiprod} (but with applying the homomorphism $\pi_n$).  	So, if  $\pi_n((k+1:1;x_k))\in \pi_n(TR_n) \subseteq\pi_n(\mathcal{MT}_n)$, 	$(\pi_n((k+1:1;x_k)))^{-1}=\pi_n((k+1:1;x_k)^{-1})= \pi_n((k+1:1;-x_k))$. 
		Consider the element 	$\pi_n(k :1;x_{k-1})\in \pi_n(\mathcal{ML}^n_{k})$,  we have since $\pi_n$ is a homomorphism,
		\begin{align*}
			\pi_n((k+1:1;x_k)(k :1;&x_{k-1})(k+1:1;-x_k))\\	 
			&=\pi_n((x_1,\ldots,x_{k-1},x_{k-1}+x_{k},x_{k-1}+x_{k+1},x_{k+2},\ldots,x_n))\\
			&= \pi_n((k :1;x_{k-1}))\pi_n((k+1:1; x_{k-1})).
		\end{align*} 
		Thus, since $\pi_n((k :1;x_{k-1}))\in \pi_n(\mathcal{ML}^n_{k})$, $\pi_n((k :1;x_{k-1}))\pi_n((k+1:1; x_{k-1}))\notin \pi_n(\mathcal{ML}^n_{k})$ if and only if $\pi_n((k+1:1; x_{k-1}))\notin \pi_n(\mathcal{ML}^n_{k})$. Let us assume that $\pi_n((k+1:1; x_{k-1}))\in \pi_n(\mathcal{ML}^n_{k})$.
		By the definition of  $\pi_n(\mathcal{ML}^n_{k})$,
		there exist $\vec{f} \in \ \mathcal{ML}^n_{k}$ such that  $\pi_n(\vec{f})=\pi_n((k+1:1; x_{k-1}))$.
		Hence, there exist $f\in R_{[k-1]}$ and $u\in UV(R_{[k-1]})$ such that $\vec{f}=(k:u;f)$ by the definition of the monoid  $\mathcal{ML}^n_{k}$. But, then by Lemma~\ref{equivrela}, $(k:u;f)\quv(k+1:1; x_{k-1})$ on $R^n$, which implies that $x_{k+1}\quv  x_{k+1}+x_{k-1}$ on $R$ by Lemma~\ref{uniqex}.
		Thus, $x_{k-1}\quv 0$ on $R$, that is the polynomial $x_{k-1}$ maps $R^n$ to the zero, which is impossible since $R$ contains $1$. Therefore, $\pi_n((k+1:1; x_{k-1}))\notin \pi_n(\mathcal{ML}^n_{k})$
		and $\pi_n(\mathcal{ML}^n_{k})$ is not normal subgroup of $\pi_n(\mathcal{MT}_n)$.
	\end{proof}
	Similar to Theorem~\ref{decoftg}, we view  the group  $\pi_n(\mathcal{MT}_n)$ as
	iterated	semi-direct products of groups. Also, we obtain a counting formula for its  order. 

	\begin{theorem}\label{ddirec}
		Let $n>1$ and let $R$ be a finite commutative ring. Then \begin{enumerate}
			
			\item 	$\pi_n(\mathcal{MT}_n)  \cong  \pi_n(\mathcal{ML}^n_{n}) \rtimes \cdots  \rtimes \pi_2(\mathcal{ML}^2_{2})\rtimes  \PrPol$; \label{PP1}

			\item $	\pi_n(\mathcal{MT}_n) \cong   ( \PolFun[R^{n-1}]  \rtimes \UPF[R^{n-1}])  \rtimes \cdots \rtimes (\PolFun   \rtimes \UPF)\rtimes \PrPol$; \label{PP2}
			\item  $ |\pi_n(\mathcal{MT}_n)| =| \PrPol|\times \prod_{i=1}^{n-1}| \PolFun[R^{i}]|| \UPF[R^{i}]|. $ \label{PP3}
		\end{enumerate}
	\end{theorem}
	\begin{proof}
		(\ref{PP3})  follows  from (\ref{PP2}).
		(\ref{PP2})   follows  from (\ref{PP1}) and Proposition~\ref{imsemiprod}-(\ref{p2}).
		By  Proposition~\ref{imsemiprod}-(\ref{p3}), $	 \pi_n(\mathcal{MT}_n)\cong      \pi_n(\mathcal{ML}^n_{n})  \rtimes  \pi_n(\mathcal{MT}_{n-1})$. Then one can use this relation and induction to prove (\ref{PP1}).
		We note to the reader that $\mathcal{MT}_1=\mathcal{MP}(R)$  and $\pi_1(\mathcal{MP}(R))=\PrPol$.
	\end{proof}

	The second decomposition of the previous theorem will allow obtaining more information about the group $\pi_n(\mathcal{MT}_n)$. 
	However, without loss of generality,  we restrict ourselves  to  finite local rings.
	We will need some facts about polynomial functions over finite local rings.
	In the proof of the following proposition, we use the fact that the order of  a finite local  commutative ring is a power of $q$, where $q$ is the number of elements of its residue field $\mathbb{F}_q$.
	\begin{proposition}\label{p-count} 
		Let $R$ be a finite local ring and $k\ge 1 $. Then 
		the group $(\PolFun[R^k],``+")$ is a $p$-group.
	\end{proposition}
	\begin{proof}
		Let $F\in \PolFun[R^k]$, then there exists a polynomial $f_F\in R_{[k]}$ of minimal degree representing $F$.
		Since $\PolFun[R^k]$ is finite, there exists a positive number $n$ such that  $\deg f_F \le n$ for every $F\in \PolFun[R^k]$. Let \[ A = 
		\{ g\in R_{[k]}\colon \deg g\le n  \}.\] 
		Then $A$ is a group with respect to addition. Further, $|A|=|R|^{1+\sum_{i=1}^{k}l_i}=q^{({1+\sum_{i=1}^{k}l_i})}$, where $l_i$ is the number of monomials of degree $i$. Hence, $A$ is  a $p$-group since $q$ is a power of a prime $p$ (say $q=p^r$ for some positive integer).  Now, define  a map $\psi\colon A \longrightarrow \PolFun[{R^k}]$ as follows:
		if $F$ is the function induced by $f$, we let $\psi(f)=F$.  Clearly, $\psi$ is a homomorphism. It is also  surjective since  
		$f_F\in A$ for every $F\in \PolFun[R^k]$.  Thus, by the First Isomorphism Theorem of groups,
		\[
		\raise2pt\hbox{$A$} \big/ 
		\lower2pt\hbox{$\ker \psi$}
		\;\cong\; \PolFun[R^k]
		\text{ and } |\PolFun[R^k]|=\frac{|A|}{|\ker \psi|}.\] But, $|\ker \psi|=p^m$ for some non-negative integer $m$ (being a subgroup of a $p$-group).  Therefore, $\PolFun[R^k]$ is a $p$-group.
	\end{proof}
	\begin{theorem}\label{uniratio}
		Let $R$ be a finite local ring and  $k\ge1$. Then \[ \frac{|\UPF[R^k]|}{|\PolFun[R^k]|}=\frac{(q-1)^{q^k}}{q^{q^k}}.\]
	\end{theorem}
	\begin{proof}
			Set $A=R/M\cong \mathbb{F}_q$. Since every function $F\colon \mathbb{F}_q^k\longrightarrow  \mathbb{F}_q$ is a polynomial function, we have $|\PolFun[A^k]|=q^{{q^k}}$. Also, since $|A^\times|=(q-1)$ we have $|\UPF[R^k]|=(q-1)^{q^k}$. Since polynomial functions preserve congruences modulo $M$, there exists 
		a natural ring epimorphism $\lambda \colon \PolFun[R^k]\longrightarrow  \PolFun[A^k]$.
		Thus, \[ |\PolFun[R^k]|=|\ker \lambda||\PolFun[A^k]|=q^{q^k}|\ker \lambda|.\]
		
		Now, an element $a\in R$ is a unit in $R$ if and only if $a+M$  is a unit in $A$. Therefore, $F\in  \UPF[R^k]$ if and only if $\lambda(F)\in  \UPF[A^k]$. This yields, $\UPF[R^k]=\lambda^{-1}( \UPF[A^k]) $, whence
		\[ |\UPF[R^k]|=|\ker \lambda||\UPF[A^k]|=(q-1)^{q^k}|\ker \lambda|.\]
		Therefore,\[ \frac{|\UPF[R^k]|}{|\PolFun[R^k]|}=\frac{(q-1)^{q^k}|\ker \lambda|}{q^{q^k}|\ker \lambda|}=\frac{(q-1)^{q^k}}{q^{q^k}}.\qedhere\]
	\end{proof}
	\begin{corollary}\label{uni2gr}
		
		Let $R$ be a finite local ring that is not a field with residue field $\mathbb{F}_2$. Then $\UPF[R^k]$ is a $2$-group.
	\end{corollary}
	\begin{proof}
		By Theorem~\ref{uniratio}, 
		$$|\UPF[R^k]|=\frac{(2-1)^{2^k}}{2^{2^k}}|\PolFun[R^k]|=\frac{|\PolFun[R^k]|}{2^{2^k}}.$$
		Since $R\ne \mathbb{F}_2$, the proof of Theorem~\ref{uniratio} shows that  $|\PolFun[R^k]|>|\PolFun[\mathbb{F}_2^k]|={2^{2^k}}$.
		The result follows from  Proposition~\ref{p-count}.
			\end{proof}
	\begin{remark}\label{ratiper}
		Similar to Theorem~\ref{uniratio}, but in one variable, 
		Jiang~\cite{ratio} has obtained the following relation between the number of 
		polynomial functions and the number of polynomial permutations on  a finite local ring $R$ which is not a field:
		
		\begin{equation}\label{ratiper1}
			\frac{|\PolFun[R]|}{|\PrPol[R]|}=\frac{q!(q-1)^q}{q^{2q}}.
		\end{equation}
	\end{remark}
	\begin{lemma}\cite[ Theorem 4.1]{Gabor}\label{gab}
		Let $R$  be a finite local commutative ring with a residue field of $q$ elements $\mathbb{F}_q$. Let $\PrPol$ be the group of 
		polynomial permutations. Then 
		\begin{enumerate}
			\item $\PrPol$ is solvable if and only if $q\le 4$;
			\item  $\PrPol$ is nilpotent if and only if $q=2$;
			\item  $\PrPol$ is abelian if and only if $R=\mathbb{F}_2$. 
		\end{enumerate}
	\end{lemma}
	With these preparations, we are ready to present our main theorem of this subsection.
	\begin{theorem}\label{infucedgrproperty}
		Let 	$n\ge 1$ and let $R$  be a finite local commutative ring with a residue field of $q$ elements $\mathbb{F}_q$. Let 
		$	\pi_n(\mathcal{MT}_n)$ be the  group of permutations on $R^n$ induced by the monoid $\mathcal{MT}_n$. Then
		\begin{enumerate}
			\item $	\pi_n(\mathcal{MT}_n)$ is solvable if and only if $q\le 4$; \label{1sol}
			\item  $	\pi_n(\mathcal{MT}_n)$ is nilpotent if and only if $q=2$; \label{2nl}
			\item  $	\pi_n(\mathcal{MT}_n)$ is  abelian if and only if $n=1$ and   $R= \mathbb{F}_2$ \label{ab3}
		\end{enumerate}
	\end{theorem} 
	\begin{proof}
			In the light of Lemma~\ref{gab} we may assume that $n>1$ since $\pi_1(\mathcal{MT}_1)=\PrPol$.
		
		By Theorem~\ref{ddirec},  $$	\pi_n(\mathcal{MT}_n) \cong   ( \PolFun[R^{n-1}]  \rtimes \UPF[R^{n-1}])  \rtimes \cdots \rtimes (\PolFun   \rtimes \UPF)\rtimes \PrPol.$$
		Now,  $\PolFun[R^{i}]$ and $ \UPF[R^{i}]$ are solvable groups (being abelian groups).
		Also, extensions of solvable groups are solvable.    Therefore
		$\pi_n(\mathcal{MT}_n)$ is solvable if and only if $\PrPol$ is solvable  if and only if $q\le 4$ by Lemma~\ref{gab}. This proves (\ref{1sol}).
		
		(\ref{2nl})  Since every subgroup of a nilpotent group is nilpotent, we need only to consider the case
		$q=2$ by Lemma~\ref{gab}.  Assume that $q=2$ and without loss of  generality we consider the case $R\ne \mathbb{F}_q$. We claim that $\pi_n(\mathcal{MT}_n)$ is a $2$-group, hence it is nilpotent since $p$-groups are nilpotents. To show the claim, we notice that by means  of the above the decomposition of  $\pi_n(\mathcal{MT}_n)$, $\pi_n(\mathcal{MT}_n)$ is a $2$-group if and only if $\PrPol$, $\PolFun[R^{i}]$ and $ \UPF[R^{i}]$ are $2$-groups for $i=1,\ldots,n-1$.    Proposition~\ref{p-count} shows that $\PolFun[R^{i}]$ is a $2$-group, while $ \UPF[R^{i}]$ is a $2$-group by Corollary~\ref{uni2gr} for $i=1,\ldots,n-1$. Finally, using Equation~(\ref{ratiper1}), we see that $\PrPol$ is a $2$-group.
		
		To prove (\ref{ab3}), 
		we need only to consider the case $R= \mathbb{F}_2$ and  $n=2$ since  $\pi_2(\mathcal{MT}_2)$ 
		is embedded in $\pi_n(\mathcal{MT}_n)$ for every $n>2$ by Remark~\ref{strcembedding}.
		The routine calculations show that
		
		$$
		\pi_2((	x_1,x_1+x_2)
		)\circ	 
		\pi_2((	x_1+1,x_1+x_2))\ne \pi_2((	x_1+1,x_1+x_2) 
		)\circ\pi_2((	x_1,x_1+x_2)). \qedhere$$
	\end{proof}

	\section{A connection to the group of polynomial permutations of the ring dual numbers of several variables}\label{sec5}
	Given a commutative ring  $R $   with unity, we define the ring of dual numbers of $n\ge 1$ variables  over $R$   to be 
	the quotient ring   $R_{[n]}/ I $, where $I$ is the ideal generated by the set $\{x_ix_j\mid  i,j\in\{1,\ldots,n\}\}$. This quotient ring contains an isomorphic copy of $R$. Also, this ring can be viewed as the free commutative $R$-algebra   $\Ralfak$ with  basis $1,\alfa_1,\ldots,\alfa_n $  such that $\alfa_i$  stands for $x_i+I$ and $\alfa_i \alfa_j=0$ for $i,j=1,\ldots,n$. 
	
	To simplify   notation  we  shall use $R_n$  instead of  $\Ralfak$ (or $R_{[n]}/ I $).
	We observe here that every element $r\in R_n$ can be written uniquely as
	\[r=r_0+\sum\limits_{i=1}^{n}r_i\alfa_i  \text{ for some } r_0,\ldots,r_n\in R.\] 
	Because of this   every polynomial $g\in R_n[x]$ has a unique representation
	$$g =g_0 +\sum\limits_{i=1}^{n}g_i \alfa_i, \text{ for some } g_0,g_1,\ldots,g_n\in R[x].$$   Also, it is evident that $r_0+\sum_{i=1}^{n}r_i\mapsto (r_0,\ldots, r_n)$ establishes a one-to-one correspondence between $R_n$ and $R^{n+1}$. This correspondence motivates us
	to relate some groups of permutations of $R_n$ and $R$, at least in the finite case. 
	
	Throughout this section, since we deal with two different rings, we use the notation $\mathcal{MT}_{n+1}(R)$ for the triangular mononid over the ring $R$ to avoid any confusion with the coefficient ring.  
	
	In this section, for a   commutative ring $R$, we embed the monoid $\mathcal{MP}(R_n)$   of permutation polynomials on $R_n$   in the triangular monoid $\mathcal{MT}_{n+1}(R)$. Moreover, this  embedding leads to the embedding of the group of polynomial permutations  $\PrPol[ R_n ]$  in the group of $\pi_{n+1}(\mathcal{MT}_{n+1}(R)) $, whenever $R$  is finite. 
	
		\begin{lemma}\label{dualpoly}
		Let $k>1$, $R$ be a commutative ring and $g=g_0+\sum\limits_{i=1}^{k}g_i \alfa_i $, where 	$g_0,\ldots,g_k\in R[x]$.
		\begin{enumerate}
			\item
			If $f\in R[x]$,
			then
			\[
			f(g)=f(g_0)+\sum\limits_{i=1}^{k} g_if'(g_0)\alfa_i.	\]
			
			\item
			If $f=f_0 +\sum\limits_{i=1}^{k}f_i \alfa_i$, where   $f_0,\ldots, f_k\in R[x]$, then 
			\[			f(g)=f_0(g_0)+ \sum\limits_{i=1}^{k}(g_if_0'(g_0)+f_i(g_0))\alfa_i.	\]
		\end{enumerate}
	\end{lemma}
	\begin{proof}
		(2) Follows from (1).
		To show (1), write $f(x)=\sum\limits_{j=0}^{m}a_j x^j$. Then expanding $f(g)$ with the fact that $\alfa_i\alfa_j=0$ for all $i,j$ yields the result.
	\end{proof}
	We need  the following facts from~\cite{polysev}.
	\begin{lemma}\cite[Theorem~4.1]{polysev}\label{dualper} 
		Let  $R$ be a finite ring. Let $f =f_0+ \sum\limits_{i=1}^{k}f_i \alfa_i$, 
		where $f_0,\ldots,f_k \in R[x]$. Then $f$ is a permutation polynomial
		on $R_k$ if and only if the following conditions hold:
		\begin{enumerate}
			\item $f_0$ is a permutation polynomial on $R$;
			\item    $f_0 $ is a unit-valued polynomial on $R$.
		\end{enumerate}
		
	\end{lemma}
	
	\begin{lemma}\cite[Corollary~3.6]{polysev} \label{CHS6} 
		
		Let $f =f_0 +\sum\limits_{i=1}^{k}f_i \alfa_i$ and $g=g_0+\sum\limits_{i=1}^{k}g_i \alfa_i $, where 
		$f_0,\ldots, f_k, g_0,\ldots, g_k \in R[x]$.
		
		Then $f $ and $g$ induce the same function on $R_k$ if and only if the following 
		conditions hold:
		\begin{enumerate}
			\item
			$f_i \quv  g_i $ on $R$ for $i=0,\dots,k$;
			\item
			$f_0'\quv g_0'$ on $R$.
		\end{enumerate}
	\end{lemma}
	We conclude this section with the following embedding result.
	\begin{proposition}
		Let $n>1$ and let $R$ be a   commutative ring. 
		Then
		
		\begin{enumerate}
			\item  $\mathcal{MP}(R_n)\hookrightarrow   \mathcal{MT}_{n+1}(R)$; and whenever $R$ is finite
			\item  $\PrPol[R_n]\hookrightarrow  \pi_{n+1}(\mathcal{MT}_{n+1}(R))$.
		\end{enumerate} 
	\end{proposition}
	\begin{proof}
		(1)	Let $f=f_0+ \sum\limits_{i=1}^{n}f_i \alfa_i\in \mathcal{MP}(R_n)$, where $f_0,\ldots, f_n\in R[x]$. Then, $f_0$ is a permutation polynomial and $f_0'$ is a unit-valued polynomial by Lemma~\ref{dualper}. Define $ \psi\colon \mathcal{MP}(R_n) \longrightarrow \mathcal{MT}_{n+1}(R)$ by \[\psi(f)=(f_0(x_1), 
		f_1(x_1)+x_2f_0'(x_1),\ldots, 		f_{n}(x_1)+x_{n+1} f_0'(x_1)). \] 					
		
		Since $f_0$ is a permutation polynomial and $f_0'$ is a unit-valued polynomial, $\psi(f)$ has the form (\ref{vectel}) in Theorem~\ref{constmon}. Thus,  $\psi(f)\in  \mathcal{MT}_{n+1}$ and $\psi$ is defined.
		Now, let $g=g_0 +\sum\limits_{i=1}^{n}g_i \alfa_i\in \mathcal{MP}(R_n)$, where $g_0,\ldots, g_n\in R[x]$. 
		By Lemma~\ref{dualpoly}, $f\circ g(x)=f_0(g_0(x))+\sum\limits_{i=1}^{n}(g_i(x)f_0'(g_0(x))+f_i(g_0(x))) \alfa_i$.
					
		Thus, by the definition of $\psi$ and the definition of composition, \begin{align*}
			\psi(f\circ g)&=(f_0( g_0(x_1)), 		(g_1(x_1)f_0'(g_0(x_1))+f_1(g_0(x_1)))+x_2g_0'(x_1)f_0'(g_0(x_1)),\ldots  \\				
			& \quad	\ldots,
			(g_n(x_1)f_0'(g_0(x_1))+f_n(g_0(x_1)))+x_{n+1} g_0'(x_1)f_0'(g_0(x_1)))\\
			&=\psi(f)\circ \psi(g).
		\end{align*} 
		
		Thus, $\psi$ is a homomorphism. Evidently, $\psi$    is one-to-one. 
		
		(2) Define $ \phi\colon \PrPol[R_n] \longrightarrow \pi_{n+1} (\mathcal{MT}_{n+1})$ by
		$\phi(F)=\pi_{n+1}(\psi(f))$, where $f=f_0+ \sum\limits_{i=1}^{n}f_i \alfa_i\in \mathcal{MP}(R_n)$ is a permutation polynomial inducing $F$ on $R_n$, with $f_0,\ldots, f_n\in R[x]$. Then, by Lemma~\ref{CHS6} and Lemma~\ref{equivrela}, $\phi$ is well defined. Also, $\phi$  is a homomorphism as it is a composition of homomorphisms. To show  $\phi$ is injective, consider $G\in\PrPol[R_n]$ such that $\phi(F)=\phi(G)$. So, there exists $g=g_0 +\sum\limits_{i=1}^{n}g_i \alfa_i\in \mathcal{MP}(R_n)$, where $g_0,\ldots, g_n\in R[x]$ such that $G$ is induced by $g$. Then $\phi(F)=\phi(G)$ implies
		that $\pi_{n+1}(\psi(g))=\pi_{n+1}(\psi(f))$, that is 
		\begin{align*}
			\pi_{n+1}((f_0(x_1), 
			f_1(x_1)+x_2f_0'(x_1)&,\ldots, 		f_{n}(x_1)+x_{n+1} f_0'(x_1)))=\\
			&=\pi_{n+1}((g_0(x_1), 
			g_1(x_1)+x_2g_0'(x_1),\ldots, 		g_{n}(x_1)+x_{n+1} g_0'(x_1))).
		\end{align*}
		By   Lemma~\ref{equivrela}, we have
			\[
		\begin{matrix}
			f_0(x_1)\quv g_0(x_1) \text{ on } R, \text{ and }\\
			f_{i}(x_1)+x_{i+1} f_0'(x_1)\quv g_{i}(x_1)+x_{i+1} g_0'(x_1)  \text{ on } R \text{ for } i=1,\dots,n.  
		\end{matrix}
		\] 
		Hence, by Lemma~\ref{uniqex}, $f_0'(x_1) \quv g_0'(x_1)$  and  $f_{i}(x_1) \quv g_{i}(x_1)$ for $i=1,\ldots,n$.
		Now,  the conditions of  Lemma~\ref{CHS6} are valid, and therefore,   $f$ and $g$ induce the same function on $R_n$, that is $F=G $.
		This finishes the proof.		
	\end{proof}
	
	\section{The tame monoid and related questions}\label{sec6}
	Recall from Remark~\ref{coincidestrctures} that the classical triangular group  $KTR_n$ consists of all invertible vector-polynomials of the form
	$(a_1x_1+a_1, f_1(x_1)+a_2x_2,\ldots,f_{n-1}(x_1,\ldots,x_{n-1})+ax_n)$, where  $a_0\in R$ and $a_1,\ldots, a_n\in R^\times$. It is well-known that the group $KTR_n$ and the affine Group $Aff_n $ consisting of all invertible  linear vector-polynomials generate the so-called tame group $Tm_n$. An element  of  $Tm_n$ is called tame,    and  the  tame question asks (when $R$ is a field and $n\ge 2$)  whether
	the group $GA_n$ of all invertible vector-polynomials (all R-automorphism of ${R_{[n]}})$ consists only of tame elements, that is, whether    $GA_n=Tm_n$ or not?   When $n=2$, the equality always holds; this was proved by Jung~\cite{Jung} for fields of characteristic zero,   and   was shown for positive characteristic by  Van der Kulk \cite{Kulktame2vgencase}. When  $n=3$ and $R$ is a field of characteristic zero, Shestakov and Umirbaiev ~\cite{Umirwild} showed that $GA_3\ne Tm_3$ by showing that the Nagata automorphism is not in $Tm_3$. While the case $n=3$ for positive characteristic fields and the general case  $n>3$  for any field are still open. Unfortunately, the embedding of the Nagata automorphism (by adding the   variables $x_4,\ldots,x_n$) in the group $GA_n$ is shown  by Smith~\cite{Marthastablytame} to be an element of the tame group $Tm_n$ for $n\ge 4$. We notice here that Maubach~\cite{Maubffield}  showed that $\pi_n(Tm_n)$ is just the
	group of all even permutations of $\mathbb{F}_{2^r}^n$ with $r>1$. 
	He noticed that finding an element $\vec{f}\in GA_n$ such that $\pi_n(\vec{f})$ is an odd permutation is sufficient to showing  $GA_n\ne Tm_n$ (see~\cite[Corollary 3.3]{Maubffield}).

	The case $n=1$ is  much easier as in this case  $Aut_R(R[x])=GA_1=Tm_1=\{ax+b: a,b\in R, a\ne 0\}$.
	However, when $R$ is not reduced, we have $Tm_1\ne GA_1 $ as for example $x+rx^2\in GA_1\setminus Tm_1$ for every non-zero nilpotent $r$. Then may be the tameness  problem is more trivial on non-reduced rings and  given an $n$, there always exists a non-tame invertible element. But, we could not find a reference discussing  the tameness problem on non-reduced rings. Nevertheless, this does not dampen our enthusiasm for suggesting a more general tame structure rather than the tame group.  This generalized structure allows us to modify
	the tameness problem and to ask some related questions.
	
	In the following, we define the tame monoid and give our general definition for  tame elements.  
	\begin{definition}
		Let $R$ be a commutative ring with unity. We call   the monoid  generated by the triangular monoid $\mathcal{MT}_n$ and the affine  group $Aff_n$  the tame  monoid which  we denote by $\mathcal{TAM}_n $. We call every element of $\mathcal{TAM}_n $ tame.
	\end{definition}
	Let $\mathcal{MPP}_n$ denote the monoid of vector-permutation polynomials. Then, by construction, 
	$\mathcal{MPP}_n$ contains $\mathcal{TAM}_n $. Now, we have the following general form of the tameness problem.
	\begin{question}\label{qtm}
		Let $n>1$, and let $R$ be a commutative ring with unity. Is it true that $\mathcal{MPP}_n=\mathcal{TAM}_n $?
	\end{question}
	
	\begin{remark}
		\begin{enumerate}
			\item If $n=1$, then we  always have $\mathcal{MPP}_1=\mathcal{TAM}_1$, as in this case  $\mathcal{MPP}_1  =\mathcal{MP}(R)=\mathcal{MT}_1$ and $Aff_1\subseteq  \mathcal{MT}_1$.
			\item If $R$ is an algebraically  closed  field of characteristic zero, by Proposition~\ref{algcase},  we have  $\mathcal{MT}_n=KTR_n$,  
			and  hence $\mathcal{MPP}_n=GA_n$. But, then Question~\ref{qtm} is reduced to the classical problem of the tame group.
		\end{enumerate} 
	\end{remark}
	We end our paper with two related questions that can be more sensible in the case of finite fields or non-reduced rings.
	\begin{question}
		Let $n>1$ and let $\vec{f}\in \mathcal{MPP}_n$. Then  does $\pi_n(\vec{f}) \in \pi_n(\mathcal{TAM}_n)$ implies that $\vec{f}\in \mathcal{TAM}_n$? 
	\end{question}
	
	Another result of Maubach~\cite{Maubffield}  over the finite field $\mathbb{F}_q$ with $q=2$ or $q$ is odd,
	shows that $\pi_n(Tm_n)$ is the symmetric group on the elements of $\mathbb{F}_q^n$ ($n>1$).  In this case, we  at least   know that the answer to the following question is true.

	\begin{question}
		
	Let $n>1$.	Is every vector-polynomial permutation  tamely represented?
		In terms of symbols: Let $\vec{F}\in  \pi_n(\mathcal{MPP}_n)$. Then does there exist $\vec{f} \in \mathcal{TAM}_n$ such that $\pi_{n}(\vec{f})=\vec{F}$? 
		
	\end{question} 
	\noindent {\bf Acknowledgment.}
	This work is supported by the Austrian Science Fund (FWF):P 35788-N. I want to express my sincere thanks to the referee for the numerous comments and valuable suggestions that improved the view of this article.  Also, I   would like to thank 
	Amran A. Al-Aghbari for proof reading.
	
	\bibliographystyle{plain}
	\bibliography{Vectorpol}

\begin{thebibliography}{10}

\bibitem{polysev}
Amr~Ali Abdulkader Al-Maktry.
\newblock Polynomial functions over dual numbers of several variables.
\newblock {\em J. Algebra Appl.}, 22(11):Paper No. 2350231, 2023.

\bibitem{Trgroup}
Valeriy~G. Bardakov, Mikhail~V. Neshchadim, and Yury~V. Sosnovsky.
\newblock Groups of triangular automorphisms of a free associative algebra and a polynomial algebra.
\newblock {\em J. Algebra}, 362:201--220, 2012.

\bibitem{finiteringbook}
Gilberto Bini and Flaminio Flamini.
\newblock {\em Finite commutative rings and their applications}, volume 680 of {\em The Kluwer International Series in Engineering and Computer Science}.
\newblock Kluwer Academic Publishers, Boston, MA, 2002.
\newblock With a foreword by Dieter Jungnickel.

\bibitem{nolineaper}
Joel.~V. Brawley and Robert Gilmer.
\newblock Fields that admit a nonlinear permutation polynomial.
\newblock {\em J. Algebra}, 123(1):111--119, 1989.

\bibitem{nonlinerpch}
Joel.~V. Brawley and George~E. Schnibben.
\newblock Polynomials which permute the matrices over a field.
\newblock {\em Linear Algebra Appl.}, 86:145--160, 1987.

\bibitem{AutAlgclosed}
S\l~awomir Cynk and Kamil Rusek.
\newblock Injective endomorphisms of algebraic and analytic sets.
\newblock {\em Ann. Polon. Math.}, 56(1):29--35, 1991.

\bibitem{GilmmerAut}
Robert~W. Gilmer, Jr.
\newblock {$R$}-automorphisms of {$R[X]$}.
\newblock {\em Proc. London Math. Soc. (3)}, 18:328--336, 1968.

\bibitem{Gabor}
Dalma G\"{o}rcs\"{o}s, G\'{a}bor Horv\'{a}th, and Anett M\'{e}sz\'{a}ros.
\newblock Permutation polynomials over finite rings.
\newblock {\em Finite Fields Appl.}, 49:198--211, 2018.

\bibitem{Hakutdrecks}
Keisuke Hakuta.
\newblock Permutation groups induced by {D}erksen groups in characteristic two.
\newblock {\em Acta Math. Vietnam.}, 46(1):123--132, 2021.

\bibitem{ratio}
Jianjun Jiang.
\newblock A note on polynomial functions over finite commutative rings.
\newblock {\em Adv. Math. (China)}, 39(5):555--560, 2010.

\bibitem{Jung}
Heinrich W.~E. Jung.
\newblock \"{U}ber ganze birationale {T}ransformationen der {E}bene.
\newblock {\em J. Reine Angew. Math.}, 184:161--174, 1942.

\bibitem{Leshchenkounitrfinitef}
Yuriy~Yu. Leshchenko and Vitaly~I. Sushchansky.
\newblock On the group of unitriangular automorphisms of the polynomial ring in two variables over a finite field.
\newblock {\em Algebra Discrete Math.}, 17(2):288--297, 2014.

\bibitem{Drings}
Alan Loper.
\newblock On rings without a certain divisibility property.
\newblock {\em J. Number Theory}, 28(2):132--144, 1988.

\bibitem{Maubffield}
Stefan Maubach.
\newblock Polynomial automorphisms over finite fields.
\newblock {\em Serdica Math. J.}, 27(4):343--350, 2001.

\bibitem{MaubachJscform}
Stefan Maubach and Abdul Rauf.
\newblock A new formulation of the {J}acobian conjecture in characteristic {$p$}.
\newblock {\em Colloq. Math.}, 146(1):15--30, 2017.

\bibitem{Maubdrecks}
Stefan Maubach and Roel Willems.
\newblock Polynomial automorphisms over finite fields: mimicking tame maps by the {D}erksen group.
\newblock {\em Serdica Math. J.}, 37(4):305--322, 2011.

\bibitem{Necha}
Alexander~A. Nechaev.
\newblock Polynomial transformations of finite commutative local rings of principal ideals.
\newblock 27:425--432, 1980.
\newblock transl. from 27 (1980) 885-897, 989.

\bibitem{monsemid}
William~R. Nico.
\newblock On the regularity of semidirect products.
\newblock {\em J. Algebra}, 80(1):29--36, 1983.

\bibitem{Nopercon}
Wilfried N\"{o}bauer.
\newblock Zur {T}heorie der {P}olynomtransformationen und {P}ermutationspolynome.
\newblock {\em Math. Ann.}, 157:332--342, 1964.

\bibitem{rotadv}
Joseph~J. Rotman.
\newblock {\em Advanced modern algebra}, volume 114 of {\em Graduate Studies in Mathematics}.
\newblock American Mathematical Society, Providence, RI, 2010.
\newblock Second edition [of MR2043445].

\bibitem{Umirwild}
Ivan~P. Shestakov and Ualbai~U. Umirbaev.
\newblock The {N}agata automorphism is wild.
\newblock {\em Proc. Natl. Acad. Sci. USA}, 100(22):12561--12563, 2003.

\bibitem{Umir3var}
Ivan~P. Shestakov and Ualbai~U. Umirbaev.
\newblock The tame and the wild automorphisms of polynomial rings in three variables.
\newblock {\em J. Amer. Math. Soc.}, 17(1):197--227, 2004.

\bibitem{Marthastablytame}
Martha~K. Smith.
\newblock Stably tame automorphisms.
\newblock {\em J. Pure Appl. Algebra}, 58(2):209--212, 1989.

\bibitem{Autjac}
Arno van~den Essen.
\newblock {\em Polynomial automorphisms and the {J}acobian conjecture}, volume 190 of {\em Progress in Mathematics}.
\newblock Birkh\"{a}user Verlag, Basel, 2000.

\bibitem{Autpoly1}
Arno van~den Essen, Ha~Huy Vui, Hanspeter Kraft, Peter Russell, and David Wright.
\newblock {\em Polynomial automorphisms and related topics}.
\newblock Publishing House for Science and Technology, Hanoi, 2007.
\newblock Lecture notes from the International School and Workshop (ICPA2006) held in Hanoi, October 9--20, 2006.

\bibitem{Kulktame2vgencase}
W.~van~der Kulk.
\newblock On polynomial rings in two variables.
\newblock {\em Nieuw Arch. Wisk. (3)}, 1:33--41, 1953.

\bibitem{Vanclosf}
B.~L. van~der Waerden.
\newblock {\em Modern {A}lgebra. {V}ol. {I}}.
\newblock Frederick Ungar Publishing Co., New York, german edition, 1949.
\newblock With revisions and additions by the author.

\bibitem{Jakunormgrp}
Jakub Zygad\l~o.
\newblock Remarks on a normal subgroup of {$GA_n$}.
\newblock {\em Comm. Algebra}, 39(6):1992--1996, 2011.

\end{thebibliography}
\end{document}